\documentclass[10pt]{amsart}
\usepackage[english]{babel}
\usepackage{amsmath,amsthm}
\usepackage{amsfonts}
\usepackage{enumerate}
\usepackage{graphicx}
\usepackage{thm-restate}
\usepackage{color}
\usepackage{hyperref}
\usepackage[all]{xy}
\usepackage{bm}

\usepackage[top=1.5in, bottom=1.5in, left=1.5in, right=1.5in]{geometry}

\newtheorem{thm}{Theorem}[section]

\newtheorem{lem}[thm]{Lemma}
\newtheorem{prop}[thm]{Proposition}
\newtheorem{conj}[thm]{Conjecture}
\newtheorem{claim}{Claim}

\newtheorem{innercustomthm}{}
\newenvironment{customthm}[1]
  {\renewcommand\theinnercustomthm{#1}\innercustomthm}
  {\endinnercustomthm}

\theoremstyle{definition}
\newtheorem{defn}[thm]{Definition}

\theoremstyle{remark}
\newtheorem{rem}[thm]{Remark}
\newtheorem{exam}[thm]{Example}

\numberwithin{equation}{section}
\theoremstyle{plain}

\newcommand{\spin}{\ifmmode{\rm Spin}\else{${\rm spin}$\ }\fi}
\newcommand{\spinc}{\ifmmode{{\rm Spin}^c}\else{${\rm spin}^c$}\fi}

\newcommand{\spincs}{\mathfrak s}

\newcommand{\hfhat}{\widehat{HF}}

\newcommand{\hfp}{HF^+}
\newcommand{\hfred}{HF_{\rm red}}

\newcommand{\hfkhat}{\widehat{HFK}}
\newcommand{\ared}[1]{\bm{A}_{{\rm red},\, #1}}
\newcommand{\aredp}[1]{\bm{A}'_{{\rm red},\, #1}}
\newcommand{\tp}{\mathcal{T}^+}
\newcommand{\calT}{\mathcal{T}}

\newcommand{\Xipq}{\mathbb{X}^+_{i,p/q}}

\newcommand{\Abold}{\bm{A}^+}
\newcommand{\Bbold}{\bm{B}^+}
\newcommand{\AT}{\bm{A}^T}
\newcommand{\vbold}{\bm{v}}
\newcommand{\hbold}{\bm{h}}
\newcommand{\Dbold}{\bm{D}_{i,p/q}^+}

\newcommand{\Ap}{\mathbb{A}^+}
\newcommand{\Bp}{\mathbb{B}^+}

\newcommand{\Z}{\mathbb{Z}}
\newcommand{\F}{\mathbb{F}}
\newcommand{\Q}{\mathbb{Q}}

\newcommand{\Zp}{\mathbb{Z}/p\mathbb{Z}}

\newcommand{\spliff}{\textup{Spli$\mathbb{F}$F}}

\DeclareMathOperator{\vol}{vol}

\newcommand{\floorfrac}[2]{\left\lfloor \frac{#1}{#2}\right\rfloor}
\newcommand{\ceilfrac}[2]{\left\lceil \frac{#1}{#2}\right\rceil}

\begin{document}
\title{Non-integer characterizing slopes and knot Floer homology}
\author{Duncan McCoy}
\address {Département de mathématiques\\
Université du Québec à Montréal\\
Canada}
\email{mc\_coy.duncan@uqam.ca}
\date{}
\begin{abstract}
Conjecturally, a knot in the 3-sphere has only finitely many non-integer non-characterizing slopes. We verify this conjecture for all knots with knot Floer homology satisfying certain simplicity conditions. The class of knots satisfying our notion of simplicity includes alternating knots, $L$-space knots and the vast majority of knots with at most 12 crossings. For arbitrary knots in the 3-sphere we show that almost all slopes $p/q$ with $|q|\geq 3$ are characterizing. In addition, we show that all $L$-space knots and almost $L$-space knots have infinitely many integer characterizing slopes.
\end{abstract}

\maketitle
% ----------------------------------------------------------------

\section{Introduction}
Given a knot $K$ in $S^3$, a rational number $p/q\in \Q$ is a {\em characterizing slope} for $K$, if the existence of an orientation preserving homeomorphism between $S_K^3(p/q)$ and $S_{K'}^3(p/q)$ implies that $K$ and $K'$ are isotopic. That is, if the oriented homeomorphism type of $S_K^3(p/q)$ distinguishes $K$ amongst all knots in the 3-sphere.

In general determining the exact set of characterizing slopes for a given knot is challenging and there are very few examples where this has been done. Every slope is a characterizing slope for the unknot \cite{Culler1987Cyclic, Gabai1987foliationsIII, Kronheimer2007lensspacsurgeries}, the trefoil and the figure-eight knot \cite{Ozsvath2019trefoil_char}. More recently, Sorya has exhibited an infinite class of knots for which the characterizing slopes are precisely the non-integer slopes \cite{Sorya2023satellite}. Other knots whose characterizing slopes have been studied in detail include the torus knot $T_{5,2}$ \cite{Ni2014characterizing, Ni2021torus_char_II}, $5_2$ \cite{Baldwin52}, the $(-2,3,7)$-pretzel knot \cite{McCoy2022pretzel} and various hyperbolic knots and Whitehead doubles \cite{Wakelin2023Whitehead}.

This article aims to understand the coarse structure of the set of characterizing slopes for an arbitrary knot. In particular, we are motivated by the conjecture that for any given knot almost all non-integer slopes are characterizing.

\begin{conj}\label{conj:main}
Let $K$ be a knot in $S^3$. Then there exists a constant $C=C(K)$ such that any slope $p/q$ satisfying $|q|\geq 2$ and $|p|+|q|\geq C$ is characterizing for $K$.
\end{conj}
As there are knots with infinitely many integer non-characterizing slopes \cite{Baker2018Noncharacterizing}, this conjecture cannot be relaxed to omit the $|q|\geq 2$ condition.
Conjecture~\ref{conj:main} has already been verified for torus knots \cite{McCoy2020torus_char}, hyperbolic $L$-space knots \cite{McCoy2019hyperbolic_char} and composite knots, for which all non-integer slopes turn out to be characterizing \cite{Sorya2023satellite}. As further evidence for this conjecture, Lackenby has shown that for a hyperbolic knot $K$ a slope $p/q$ is characterizing whenever $q$ is sufficiently large \cite{Lackenby2019characterizing}. Moreover, effective bounds on the meaning of sufficiently large were derived by Wakelin \cite{Wakelin2023Whitehead}. Lackenby's result was generalized by Sorya, who showed that for any knot slopes with $q$ sufficiently large are characterizing \cite{Sorya2023satellite}.

The first main result of this paper is to show that for a given knot almost all slopes with $|q|\geq 3$ are characterizing.
\begin{restatable}{thm}{qgeqthree}
\label{thm:g_geq_3}
Let $K$ be a knot in $S^3$. Then there exists a constant $C=C(K)$ such that any slope $p/q$ satisfying $|q|\geq 3$ and $|p|+|q|\geq C$ is characterizing for $K$.
\end{restatable}
This result was previously known for hyperbolic knots \cite{McCoy2019hyperbolic_char}. The main technical advance here is to show that it holds for satellite knots, which is to say, knots whose complements have a non-trivial JSJ decomposition.

Thus the veracity of Conjecture~\ref{conj:main} depends on the existence of half-integer characterizing slopes. Our second main result is to demonstrate the existence of such slopes for knots satisfying a technical knot Floer homology condition that we dub Property~\spliff{} (see Definition~\ref{def:spaff1} below). The reader should interpret Property~\spliff{} as saying that $K$ has relatively simple knot Floer homology. 
\begin{restatable}{thm}{spaffslopes}\label{thm:spaff_slopes}
Let $K$ be a knot in $S^3$ with Property~\spliff{}. Then, there exists a constant $C=C(K)\geq 0$ such that for any odd $p\geq C$, the slope $p/2$ is a characterizing slope for $K$.
\end{restatable}
Recall that there is a homeomorphism $S_K^3(-p/q)\cong -S_{mK}^3(p/q)$, where $mK$ denotes the mirror of $K$. Thus, a slope $-p/q$ is characterizing for $K$ if and only if $p/q$ is a characterizing slope for $mK$. Therefore, if $K$ is a knot such that both $K$ and $mK$ have Property~\spliff{}, then Theorem~\ref{thm:g_geq_3} and Theorem~\ref{thm:spaff_slopes} show that $K$ satisfies Conjecture~\ref{conj:main}. Using this, we establish Conjecture~\ref{conj:main} for several classes of knots, notably including alternating knots and $L$-space knots.  Furthermore, for the 2977 prime knots in the tables with at most 12 crossings, Conjecture~\ref{conj:main} holds for at least 2951 of them.
\begin{restatable}{cor}{interestingclasses}\label{cor:interesting_classes}
Let $K$ be a knot satisfying at least one of the following conditions:
\begin{enumerate}[(i)]
\item $K$ has thin knot Floer homology;
\item $K$ is an $L$-space knot; or
\item $K$ is a prime knot with at most 12 crossings and up to mirroring is not amongst the 26 knots listed in Table~\ref{table:awkward_examples}.
\end{enumerate}
Then there exists a constant $C=C(K)$ such that any slope $p/q$ satisfying $|q|\geq 2$ and $|p|+|q|\geq C$ is characterizing for $K$.
\end{restatable}

\begin{table}[h!]
\centering
\begin{tabular}{|p{0.8\linewidth}|}
\hline
$11n27$, $11n81$, $11n126$, $12n67$, $12n68$, $12n89 $, $12n93$, $12n96 $, $12n134$, $12n138$, $12n148$, $12n153 $, $12n203 $, $12n229$, $12n244$, $12n276$, $12n292$, $12n368 $, $12n386 $, $12n402 $, $12n418 $, $12n503 $, $12n638$, $12n639$, $12n644$, $12n764$\\
\hline
\end{tabular}
\caption{Knots with at most 12 crossings for which Conjecture~\ref{conj:main} is unresolved.}\label{table:awkward_examples}
\end{table}

\subsection{Property~\spliff{}}
We give a quick definition of Property~\spliff{} for readers familiar with knot Floer homology. Alternative formulations and a more detailed explanation of terminology will be given in Section~\ref{sec:spaff}. Given a knot $K$ with knot Floer chain complex $C=CFK^\infty(K)$ (with $\F=\Z/2\Z$ coefficients), we take $\Abold_k$ to be the homology of the quotient complex 
\[A_k^+=C\{i\geq 0 \text{ or } j\geq k\}.\]
Recall that $\Abold_k$ has the structure of a $\Z$-graded $\F[U]$-module. We use $\F_{(d)}$ to denote the graded $\F[U]$-module consisting of a copy of $\F$ supported in grading $d$ with the module structure given by $ U \F_{(d)}=0$.
\begin{defn}\label{def:spaff1}
We say the knot $K$ has \emph{split $\F$-factors} (Property~\spliff{}) if for all $k\in \Z$ the graded $\F[U]$-module $\Abold_k$ admits a direct sum decomposition of the form
\[
\Abold_k\cong A'\oplus \F_{(d_1)}^{n_1} \oplus \F_{(d_2)}^{n_2},
\]
where $n_1,n_2\geq 0$, $d_1$ is odd, $d_2$ is even and the module $A'$ contains no further summands of the form $\F_{(d)}$ for any $d$.
\end{defn}
This definition says that for each $\Abold_k$, all $\F_{(d)}$ summands are supported in at most two gradings one of which is odd and the other is even. 

The following proposition will be our main tool for exhibiting knots with Property~\spliff{}. 

\begin{restatable}{prop}{propthicknessone}\label{prop:thickness_one}
Let $K$ be a knot whose knot Floer homology has thickness at most one and let $\rho\in \Z$ be a constant such that $\hfkhat_d(K,s)$ is non-zero only if 
\[d\in \{s+\rho,s+\rho-1\}.\]
\begin{enumerate}[(i)]
\item If $\rho\leq 2$ or at least one of the groups $\hfkhat_{2\rho-3}(K,\rho-3)$ or $\hfkhat_{2\rho-4}(K,\rho-3)$ is trivial, then $K$ has Property~\spliff{}.
\item If $\rho\geq -1$ or at least of one of the groups $\hfkhat_{2\rho+2}(K,\rho+2)$ or $\hfkhat_{2\rho+1}(K,\rho+2)$ is trivial, then $mK$ has Property~\spliff{}.
\end{enumerate}
\end{restatable}
Since the knots with thin knot Floer homology are, by definition, those such that there exists a $\rho$ such that $\hfkhat_d(K,s)$ is non-zero only if $d=s+\rho$, Proposition~\ref{prop:thickness_one} easily implies that knots with thin knot Floer homology have Property~\spliff{}. Furthermore, with the exception of those listed in Table~\ref{table:awkward_examples}, Proposition~\ref{prop:thickness_one} shows that for all prime knots with at most 12 crossings both $K$ and $mK$ have Property~\spliff{}.

\subsection{Integer characterizing slopes}
Although this article primarily focused on non-integer characterizing slopes, we are also able to expand the class of knots known to have integer characterizing slopes. Previously, it was known that torus knots \cite{Ni2014characterizing} and hyperbolic $L$-space knots \cite{McCoy2019hyperbolic_char} have infinitely many integer characterizing slopes. We expand this class to include all $L$-space knots and almost $L$-space knots. An $L$-space knot is one whose surgeries have the smallest possible Heegaard Floer homology groups, which is to say
\[
\dim \hfhat(S_K^3(n))= n
\]
for some (equivalently any) integer $n$ satisfying $n\geq 2g(K)-1$ \cite{Ozsvath2005Lensspace}.  The notion of an almost $L$-space knot was introduced by Baldwin and Sivek and describes the knots whose surgeries have the next smallest possible Heegaard Floer homology groups after $L$-space knots \cite{Baldwin52}. More precisely, a knot $K$ in $S^3$ is an \emph{almost $L$-space knot} if
\[
\dim \hfhat(S_K^3(n))= n+2.
\]
for some (equivalently any) integer $n$ satisfying $n\geq 2g(K)-1$.

\begin{restatable}{thm}{almostLspaceknot}\label{thm:almostlspaceknot}
Let $K$ be an $L$-space knot or an almost $L$-space knot in $S^3$. Then there exists a constant $C=C(K)$ such any slope $p/q>C$ is characterizing for $K$.
\end{restatable}
One interesting example of an almost $L$-space knot is the composite knot $T_{3,2}\# T_{3,2}$. In fact, this is the only composite almost $L$-space knot \cite{Binns2019almostLspace}. In this case we obtain explicit bounds on the set of characterizing slopes for $T_{3,2}\# T_{3,2}$ and show that any integer slope $n\geq 32$ is a characterizing slope for $T_{3,2}\# T_{3,2}$. As far as the author is aware, this is the first composite knot known to possess integer characterizing slopes. In contrast, Varvarezos has shown that there exists connected sums of torus knots with only finitely many integer characterizing slopes \cite{Varvarezos2023torus}.

\subsection{Proof of Theorems~\ref{thm:g_geq_3} and~\ref{thm:spaff_slopes}}
We conclude the introduction by proving Theorem~\ref{thm:g_geq_3} and Theorem~\ref{thm:spaff_slopes} assuming Theorem~\ref{thm:bounded_genus} and Theorem~\ref{thm:HFKbound}, which are two technical results whose proofs occupy Section~\ref{sec:JSJ} and Section~\ref{sec:spaff}, respectively. Since both Theorem~\ref{thm:g_geq_3} and Theorem~\ref{thm:spaff_slopes} follow the same strategy we proceed with both simultaneously. Let $K$ be a knot in $S^3$ and suppose that $p/q\in \Q\setminus\Z$ is a non-integer non-characterizing slope for $K$.  Since $p/q$ is a characterizing slope for $K$ if and only if $-p/q$ is a characterizing slope for $mK$, we may assume that $p/q>0$. Thus, let $K'$ be a knot such that $K'\not\simeq K$ and $S_{K}^3(p/q)\cong S_{K'}^3(p/q)$. The is aim to obtain upper bounds on $p$ and $q$ that are independent of the particular $K'$ under consideration. Two recent results of Sorya will significantly aid us in our endeavours. Firstly, Sorya has already bounded $q$.
\begin{thm}[Sorya, \cite{Sorya2023satellite}]\label{thm:Sorya1}
Let $K$ be a knot in $S^3$. Then there exists a constant $Q=Q(K)$ such that any slope $p/q$ satisfying $|q|\geq Q$ is a characterizing slope for $K$. \qed
\end{thm}

 Secondly, as Sorya has shown that every non-integer is characterizing for a composite knot \cite{Sorya2023satellite}.
\begin{thm}[Sorya, \cite{Sorya2023satellite}]\label{thm:Sorya2}
Let $K$ be a composite knot in $S^3$. Any non-integer slope $p/q\in \Q\setminus \Z$ is characterizing for $K$. \qed
\end{thm}
This allows us to further assume that $K$ and $K'$ are both prime knots.

We use two different techniques to bound the numerator $p$ depending on the genus of the knot $K'$. When the genus $g(K')$ is small, which is to say $g(K')\leq g(K)$, we are able to gain control over $p$ by carefully comparing the JSJ decompositions of the complements of the knots $K$ and $K'$. This results in the following bound.
\newcommand{\thmboundedgenus}{Let $K$ be a prime knot in $S^3$. Then there exists a constant $M=M(K)$ such that if $K'\not\simeq K$ is a prime knot with $g(K')\leq g(K)$ and we have a homeomorphism $S_{K}^3(p/q)\cong S_{K'}^3(p/q)$ for some $p/q\in \Q$, then $|p|\leq M|q|$.}

\begin{customthm}{Theorem~\ref{thm:bounded_genus}}
\thmboundedgenus{}
\end{customthm}

The second method is to study the Heegaard Floer homology of surgeries on $K$ and $K'$. This complements Theorem~\ref{thm:bounded_genus} by yielding upper bounds on $p$ when the genus of $K'$ is relatively large, which is to say $g(K')>g(K)$. The theorem stemming from this method is the following. The conclusions pertaining to the Alexander polynomial and fibredness do not play a role in this paper, but are included as they may admit future applications.
\newcommand{\thmHFKbound}{Let $K$ be knot in $S^3$. Suppose that there is a knot $K'$ and a slope $p/q>0$ such that $p\geq 12+4q^2 -2q +4qg(K)$ and $S_{K}^3(p/q)\cong S_{K'}^3(p/q)$. If one of the following conditions holds:
\begin{enumerate}[(i)]
\item $q\geq 3$,
\item $q=2$ and $K$ has Property~\spliff{}.
\end{enumerate}
Then $g(K)=g(K')$, $\Delta_K(t)=\Delta_{K'}(t)$ and $K$ is fibred if and only if $K'$ is fibred.}

\begin{customthm}{Theorem~\ref{thm:HFKbound}}
\thmHFKbound{}
\end{customthm}
We note that the part of Theorem~\ref{thm:HFKbound} pertaining to the case $q\geq 3$ was previously established in \cite[Theorem~1.7]{McCoy2020torus_char}. The content of this paper is the result valid for $q=2$.
Thus if we assume that $q\geq 3$ or $q=2$ and $K$ has Property~\spliff{}, then Theorem~\ref{thm:bounded_genus} and Theorem~\ref{thm:HFKbound} together imply that
\begin{align*}
p&\leq \max\{ Mq, 12+4q^2-2q+4qg(K)\}\\
&\leq \max\{ MQ, 12+4Q^2+4Qg(K)\},
\end{align*}
where $Q$ and $M$ are the constants from Theorem~\ref{thm:Sorya1} and Theorem~\ref{thm:bounded_genus} respectively. Since we have bounds on both $p$ and $q$, Theorems~\ref{thm:g_geq_3} and~\ref{thm:spaff_slopes} then follow.

\begin{rem}
Our inability to prove Conjecture~\ref{conj:main} for all knots stems from the knot Floer homology component  of the paper. We are unable to extract the necessary constraints on $p$ from Heegaard Floer homology when $q=2$.
 \end{rem}

\subsection*{Acknowledgements}
The author would like to thank Steve Boyer, Jen Hom and Patricia Sorya for helpful conversations; Fraser Binns for communicating an early version of his paper \cite{Binns2019almostLspace} and the anonymous referee for their thoughtful feedback. This work was carried out whilst the author was supported by FRQNT, NSERC and a Canada Research Chair.

\section{Surgery on knots of bounded genus}\label{sec:JSJ}
Given a knot $K$ in a 3-manifold $M$, we will use $M_K$ to denote the complement of the knot. That is, $M_K$ denotes $M$ minus an open tubular neighbourhood of $K$. In particular, given a knot $K$ in $S^3$, the complement of $K$ will be denoted be $S^3_K$. Given a 3-manifold $M$ with a torus boundary component $T$, a slope $\sigma$ on $T$ is an unoriented essential simple closed curve on $T$ considered up to isotopy. Given a 3-manifold $M$ and a slope $\sigma$ on a torus boundary component of $M$ we will use $M(\sigma)$ to denote Dehn filling along that slope. For a knot complement $S^3_K$, we have a natural basis for $H_1(\partial S_K^3;\Z)$ coming from the meridian $\mu$ and null-homologous longitude $\lambda$. This allows us to identify slopes on $\partial S_K^3$ with $\Q \cup \{1/0\}$ in the usual way: a slope $\sigma$ corresponds to $p/q$ if  it is homologous to $p\mu +q \lambda$ for some choice of orientation on $\sigma$.   Thus $p/q$-surgery along a knot $K$ in $S^3$ is denoted by $S^3_K(p/q)$. In the event that we are filling a manifold with multiple boundary components we will not overburden the notation by explicitly indicating which boundary component is being filled since this will always be clear from context.

We remind the reader that any compact, orientable, irreducible 3-manifold $M$ whose boundary is a (possibly empty) collection of tori admits a finite collection of embedded tori which decompose $M$ into pieces which are either atoroidal or Seifert fibred. The minimal collection of such tori is unique up to isotopy \cite{Jaco1978decomposition, Johannson1979homotopy}. Splitting $M$ along such a minimal collection of tori yields the JSJ decomposition of $M$.

 For a knot $K$ in $S^3$ we will need to consider the JSJ decomposition of the knot complement $S^3_K$. Given the JSJ decomposition of $S_{K}^3$, then one its components is distinguished as containing the boundary of $S_{K}^3$. We will refer to this piece as the {\em outermost piece} of the decomposition. The outermost piece has a distinguished slope corresponding to the meridian of the knot. Budney provides a detailed description of the possible components arising in the JSJ decomposition of a knot exterior \cite{Budney2006JSJ}.
\begin{thm}\label{thm:knot_JSJ_decomp}
Let $K$ be a non-trivial knot in $S^3$ and let $X$ be a JSJ component of $S^3_K$. Then $X$ takes one of the following forms.
\begin{enumerate}
    \item $X$ is a Seifert fibred space over the disk with two exceptional fibres. If the JSJ decomposition of $S^3_K$ contains such an $X$, then it is necessarily the outermost piece and $K$ is a torus knot.
    \item $X$ is homeomorphic to $S^1\times F$ where $F$ is a compact planar surface with at least three boundary components. If such an $X$ is the outermost component, then $K$ is a composite knot.
    \item $X$ is a hyperbolic manifold. If such an $X$ is the outermost piece, then filling $X$ along the meridian yields a (possibly empty; i.e homeomorphic to $S^3$) connected sum of solid tori.
    \item $X$ is a Seifert fibred space over the annulus with one exceptional fibre. If such an $X$ is the outermost piece, then $K$ is a cable of a non-trivial knot $K'$.
\end{enumerate}
\end{thm}
\begin{proof}
This is a subset of the information contained in \cite[Theorem~4.18]{Budney2006JSJ}.
\end{proof}
For brevity, we will refer to a knot $K$ for which the outermost piece of the JSJ decomposition is hyperbolic as a knot of \emph{hyperbolic type}. By definition, every hyperbolic knot in $S^3$ is an example of a knot of hyperbolic type.

\begin{rem} We note some consequences of Theorem~\ref{thm:knot_JSJ_decomp}.
\begin{itemize}
\item If a knot is prime, then it is either a torus knot, a cable knot or a knot of hyperbolic type.
\item Since the only compact 3-manifolds with boundary admitting more than one Seifert fibred structure are the solid torus and the twisted $I$-bundle over the Klein bottle \cite[Theorem~2.4.32]{Brin2007seifertfiberedspacesnotes}, all of the Seifert fibred spaces occurring in Theorem~\ref{thm:knot_JSJ_decomp} admit unique Seifert fibred structures. Thus we can unambiguously refer to the orders of the exceptional fibres in these spaces are invariants of the knot $K$.
\end{itemize}
\end{rem}

\subsection{Knots of hyperbolic type}
In this section we analyse surgeries on knots of hyperbolic type, which are those for which the outermost piece of the JSJ decomposition is a hyperbolic manifold.

Given a knot $K$ of hyperbolic type, we can assign a length to slopes on $\partial S_K^3$ as follows. Let $X_0$ be the outermost piece in the JSJ decomposition of $S^3_{K}$; this is, by definition, hyperbolic. Given a maximal horocusp neighbourhood $N$ of the cusp corresponding to $\partial S^3_{K} \subseteq X_0$, we can measure the length of a slope $p/q$ on $\partial S^3_{K}$ using the natural Euclidean metric on $\partial N$. We take $\ell_K(p/q)$ to denote the length of $p/q$ when measured in the maximal horocusp neighbourhood.

By appropriately manipulating a Seifert surface, we can guarantee that it meets the outermost piece of the JSJ decomposition nicely.
\begin{prop}\label{prop:hyperbolic_surface}
Let $K$ be non-trivial knot in $S^3$. Then the outermost piece of the JSJ decomposition of $S_{K}^3$ contains an essential embedded surface $S$ which meets $\partial S_{K}^3$ in a longitude and whose Euler characteristic satisfies $|\chi(S)|\leq 2g(K)-1$.
\end{prop}
\begin{proof}
\setcounter{claim}{0}
Let $X_0$ be the outermost piece of the JSJ decomposition for $S_{K}^3$. Take $\Sigma$ to be a minimal genus Seifert surface for $K$ which intersects $\partial X_0$ transversely in the minimal possible number of curves. We will show, using standard 3-manifold arguments, that $S=\Sigma \cap X_0$ is a surface with the desired properties. If $K$ is a simple knot, then $S=\Sigma$ and the condition $|\chi(S)|\leq 2g(K)-1$ is automatically satisfied. 

\begin{claim}
The surface $S$ is incompressible.
\end{claim}
\begin{proof}[Proof of Claim]
Let $D_1\subseteq X_0$ be a compressing disk for $S$, i.e. an embedded disk with $\partial D_1 =D_1\cap S$. Since the surface $\Sigma$ is incompressible in $S_K^3$, there must be an embedded disk $D_2\subseteq \Sigma$ with $\partial D_1=\partial D_2$. We may obtain a new Seifert surface $\Sigma'$ by replacing the interior of $D_2$ with the interior of $D_1$. If $D_2 \cap \partial X_0$ were non-empty, then $\Sigma'\cap X_0$ would contain fewer components than $\Sigma\cap X_0$. However, we are assuming $\Sigma\cap X_0$ to be minimal and so $D_2 \cap \partial X_0$ must be empty. This implies that $D_2$ is contained in $X_0$ and hence that $D_2\subseteq S$. This verifies the incompressibility of $S$.
\end{proof}

\begin{claim}
The surface $S$ is boundary incompressible.
\end{claim}
\begin{proof}[Proof of Claim]
Let $D_1\subseteq X_0$ be a $\partial$-compressing disk for $S$, i.e. an embedded disk whose boundary can be decomposed into two intervals $\partial D_1=\alpha \cup \beta$ such that $\beta =D_1\cap S$ and $\alpha=D_1\cap \partial X_0$. There is an isotopy of $\Sigma$ guided by the disk $D_1$ to a surface $\Sigma'$ such that $\Sigma' \cap \partial X_0$ is obtained from $\Sigma \cap \partial X_0$ by ambient surgery along the arc $\alpha$ in $\partial X_0$. Thus if the arc $\alpha$ connected two distinct components of $\Sigma \cap \partial X_0$, then $\Sigma' \cap \partial X_0$ would contain fewer components than $\Sigma \cap \partial X_0$. This contradicts the assumed minimality of $\Sigma \cap \partial X_0$. Thus we see that the end points of $\alpha$ must lie on a single component $\gamma$ of $\Sigma \cap \partial X_0$.

 Let $T\subseteq \partial X_0$ be the boundary component containing $\gamma$ and $\alpha$. Since $S$ is a 2-sided surface in $X_0$, the curve $\alpha$ separates $T\setminus \gamma$ into two components. One of these components is an disk $D_2$ whose interior is necessarily disjoint from $S$. The disk $D_1\cup D_2$ forms a compressing disk for $S$. As $S$ is incompressible, the disk $\partial D_1 \cup D_2$ bounds a disk in $S$. Thus the original disk $D_1$ cuts off a disk from $S$. This shows that $S$ is boundary incompressible.
\end{proof}
Thus we have shown that the surface $S$ is essential and by construction it evidently meets $\partial S_K^3$ in a longitude. Thus, it remains only to consider the Euler characteristic of $S$. Since the boundary tori of $\partial X_0$ are incompressible in $S^3_{K}$, none of the curves in $S\cap \partial X_0$ bound a disk in $S^3_{K}$. It follows that no component of $\Sigma \setminus S$ can be a disk. In particular, $\chi(\Sigma\setminus S)\leq 0$.
Thus, thus we obtain the bound:
\[
\chi(S)\geq \chi(\Sigma \setminus S)+ \chi(S)=\chi(\Sigma)=1-2g(K),
\]
which is to say $|\chi(S)|\leq 2g(K)-1$.
\end{proof}

Using Proposition~\ref{prop:hyperbolic_surface}, we can obtain bounds on $\ell_K(p/q)$ in terms of $g(K)$.

\begin{prop}\label{prop:length_bound}
Let $K$ be a knot of hyperbolic type. Then 
\[
|p|\leq \sqrt{3}(2g(K)-1)\ell_K(p/q).
\]
\end{prop}

\begin{proof}
Let $X_0$ be the outermost piece in the JSJ decomposition of $S^3_{K}$. By the classification of hyperbolic manifolds with minimal cusp volume \cite[Theorem~1.2]{Gabai_low_vol}, there exists a horocusp neighbourhood $N$ of $\partial S_{K}^3\subseteq X_0$ with $\mathrm{Area}(\partial N)\geq 2\sqrt{3}$.
A simple geometric argument in the universal cover of $\partial N$ (e.g. as used by Cooper and Lackenby \cite[Lemma~2.1]{Cooper1998Dehn}) shows that for any pair of slopes $\alpha$ and $\beta$ on $\partial S_K^3$ we have 
\begin{equation}\label{eq:length_bound1}
2\sqrt{3}\Delta(\alpha,\beta)\leq \ell_K(\alpha)\ell_K(\beta)
\end{equation}
where $\Delta(\alpha,\beta)$ denotes the distance between $\alpha$ and $\beta$. Given the essential surface $S$ constructed in Proposition~\ref{prop:hyperbolic_surface}, we can apply results of Agol to obtain the bound  \cite[Theorem~5.1]{Agol2000BoundsI} (or for a similar bound, see \cite[Theorem~5.1]{Cooper1998Dehn}):
\[\ell_K(0/1)\leq 6(2g(K)-1).\] 
Since $\Delta(0/1,p/q)=|p|$,  applying \eqref{eq:length_bound1} to the slopes $0/1$ and $p/q$ yields
\[
2\sqrt{3}|p|\leq 6(2g(K)-1)\ell_K(p/q).
\]

\end{proof}

We make also the following observation.
\begin{lem}\label{lem:fixed_JSJ_piece}
Let $M$ be any 3-manifold. Then there exists a constant $L=L(M)$ depending only on $M$ such that if $N$ is a hyperbolic 3-manifold and $\sigma$ a slope on a component of $\partial N$ such that $N(\sigma)\cong M$, the slope $\sigma$ has length $\ell(\sigma)\leq L$ as measured in any horocusp neighbourhood.
\end{lem}
\begin{proof}
If $M$ is not a hyperbolic manifold, then the 6-theorem yields the bound $\ell(\sigma)\leq 6$ \cite{Agol2000BoundsI, Lackenby2003Exceptional}. Thus we may assume that $M$ is itself a hyperbolic manifold. Since hyperbolic volume strictly decreases under Dehn filling, we have that $\vol(N)>\vol(M)$ \cite{Thurston_notes}. As the volumes of hyperbolic 3-manifolds form a well-ordered set \cite{BenedettiPetronio}, there exists $\varepsilon>0$ such that $\vol(N)\geq \vol(M)+\varepsilon$. Now suppose that the slope $\sigma$ has length $\ell>2\pi$. By the work of Futer-Kalfagianni-Purcell \cite[Theorem~1.1]{Futer2008Dehn_filling} we have that
\[
\vol(N)\left(1- \left(\frac{2\pi}{\ell}\right)^2\right)^{\frac32}\leq \vol(M).
\]
Thus we have that
\[
\left(1- \left(\frac{2\pi}{\ell}\right)^2\right)^{\frac32}\leq \frac{\vol(M)}{\vol(M)+\varepsilon}
\]
from which follows an upper bound on $\ell$ in terms of $\vol(M)$ and $\varepsilon$ which are independent of $N$.
\end{proof}

We will also make use of the following variation on a result of Lackenby \cite[Theorem~3.1]{Lackenby2019characterizing}. 
\begin{thm}\label{thm:Lackenby_magic}
Let $M$ be $S^3$ or the exterior of an unknot or unlink in $S^3$ and let $J$ be a hyperbolic knot in $M$. Then there is a constant $L(J)$ such that the following property holds. Let $J'\subseteq M$ be a hyperbolic knot and let $\sigma'$ and $\sigma$ be slopes for $J'$ and $J$ respectively such that there is a homeomorphism $f:M_{J'}(\sigma')\rightarrow M_J(\sigma)$. If $\ell_{J'}(\sigma')>L$, then there exists a homeomorphism $f':M_{J'}\rightarrow M_J$ such that $f'|_{\partial M}=f|_{\partial M}$.\qed
\end{thm}
We omit the proof as it follows the one given in \cite{Lackenby2019characterizing} with only cosmetic modifications. The key point is that the condition on $q'$ in \cite[Theorem~3.1]{Lackenby2019characterizing} is simply to ensure that Dehn fillings occur along sufficiently long slopes. Note also that the conclusion of Theorem~\ref{thm:Lackenby_magic} in terms of the homeomorphisms $f$ and $f'$ is simply the result of unpacking the notation $\cong_\partial$ in \cite{Lackenby2019characterizing}. 

\subsection{Surgery and the JSJ decomposition}
We wish to understand how the JSJ decomposition of $S^3_K$ changes under Dehn filling. If the JSJ decomposition of $S_K^3$ takes the form
\[
S_K^3=X_0 \cup X_1 \cup \dots \cup X_n,
\]
where $X_0$ is the outermost piece of this decomposition, then for any $p/q\in \Q$ the surgered manifold $S_K^3(p/q)$ can be decomposed as 
\[
S_K^3(p/q)=X_0(p/q) \cup X_1 \cup \dots \cup X_n,
\]
where $X_0(p/q)$ denotes $X_0$ filled along the slope $p/q$ on $\partial S_K^3\subseteq \partial X_0$. We wish to understand when this decomposition corresponds to the JSJ decomposition of $S_K^3(p/q)$. We do this first when $K$ is a torus knot or a cable knot.

\begin{prop}\label{prop:JSJ_surgery_torus}
Let $K$ be a torus knot in $S^3$. If $|p/q|> 4g(K)+4$, then $S_{K}^3(p/q)$ is a Seifert fibred space with a unique Seifert fibred structure. Moreover, the core of the filling torus can be assumed to be an exceptional fibre of order at least $|p/q|-4g(K)-2$ in this structure.
\end{prop}
\begin{proof}
Suppose that $K$ is the $(r,s)$-torus knot, where $|r|>|s|\geq 2$. If $|p-rsq|\geq 2$, then $S_{K}^3(p/q)$ admits the structure of a Seifert fibred space over $S^2$ with exceptional fibres of order $|r|$, $|s|$ and $|p-rsq|$ \cite{Moser1971elementary}. Moreover, this Seifert fibred structure is such that the core of the filling solid torus is an exceptional fibre of $|p-rsq|$.

Since the genus of $K$ is
\[
g(K)=\frac{(|r|-1)(|s|-1)}{2},
\]
we can bound $|p-rsq|$ below as follows:
\begin{align*}
|p-rsq|&\geq |p/q-rs|\\
&\geq |p/q|- |rs|\\
&= |p/q| -4g(K) -2+ (|r|-2)(|s|-2)\\
&\geq |p/q| -4g(K) -2.
\end{align*}

Thus if $|p/q|> 4g(K)+4$, we have that $|p-rsq|> 2$ and so the Seifert fibred structure on $S_{K}^3(p/q)$ is unique since it has at least two exceptional fibres of order three.
\end{proof}
A similar analysis applies to cables.
\begin{prop}\label{prop:JSJ_surgery_cable}
Let $K$ be a non-trivial cable knot in $S^3$ and suppose that the JSJ decomposition of $S_K^3$ is
\[
S_K^3=X_0 \cup X_1 \cup \dots \cup X_n,
\]
where $X_0$ is the outermost piece of this decomposition. If $|p/q|> 4g(K)-4$, then the JSJ decomposition of $S_K^3(p/q)$ is
\begin{equation}\label{eq:cable_decomp}
S_K^3(p/q)=X_0(p/q)\cup X_1 \cup \dots \cup X_n.
\end{equation}
Moreover $X_0(p/q)$ is a Seifert fibred space admitting a unique Seifert fibred structure and the core of the filling torus can be assumed to be an exceptional fibre of order at least $|p/q| - 4g(K) + 6$ in this structure.
\end{prop}
\begin{proof}
Suppose that $K$ is the $(r,s)$-cable of a non-trivial knot $J$, where $s\geq 2$ is the winding number.
By hypothesis, the outermost piece $X_0$ in the JSJ decomposition in $S_{K}^3$ is a Seifert fibred space over the annulus with an exceptional fibre of order $s$. Moreover the slope of the regular fibre on the boundary component $\partial S_{K}^3\subset X_0$ is $rs/1$. Thus if $|p-rsq|\geq 2$, then the Seifert fibred structure on $X_0$ extends to a Seifert fibred structure on $X(p/q)$ in which the core of the filling torus is an exceptional fibre of order $|p-rsq|$. Moreover if $|p-rsq|\geq 3$, then this Seifert fibred structure on $X_0(p/q)$ is unique since it is fibred over the disk with two exceptional fibres and at least one of these fibres is of order at least three.

Thus, if $|p-rsq|\geq 3$ we see that the decomposition in \eqref{eq:cable_decomp} separates $S_K^3(p/q)$ into pieces which are all atoroidal or Seifert fibred. Suppose that $X_1$ is the piece in the JSJ decomposition of $S_K^3$ which is adjacent to $X_0$. Since we started with the JSJ decomposition for $S_K^3$, the only way the above decomposition can fail to be the JSJ decomposition for $S_K^3(p/q)$ is if $X_1$ is Seifert fibred and its Seifert fibration can be extended over $X_0(p/q)$. However, this is not possible if $|p-rsq|\geq 3$, since the resulting Seifert fibration on $X_0(p/q)$ is unique and agrees with the original Seifert fibration on $\partial X_0 \setminus \partial S_K^3$. Thus \eqref{eq:cable_decomp} is the JSJ decomposition in this case.

Using the formula for the genus of a satellite knot \cite{Schubert1953knoten}, we have that
\[
g(K)=|s|g(J) +\frac{(|r|-1)(|s|-1)}{2}\geq |s|+\frac{(|r|-1)(|s|-1)}{2},
\]
where the inequality follows from the fact that $J$ is non-trivial and hence satisfies $g(J)\geq 1$. Using this, one obtains the bound
\begin{align*}
|p-rsq|&\geq |p/q| -|rs|\\
 &= |p/q| -4g(K) +4g(K)-|rs|\\
&\geq |p/q| -4g(K) +|rs|-2|r|+2|s|+2\\
&= |p/q| -4g(K) +(|r|+2)(|s|-2)+6\\
&\geq  |p/q| -4g(K) +6.
\end{align*}
Thus if $|p/q|> 4g(K)-4$, then we have that $|p-rsq|>2$ and hence the desired conclusions.
\end{proof}

Incorporating these results with the case of hyperbolic type knots we obtain the following theorem.
\begin{thm}\label{thm:surgery_JSJ}
Let $K$ be a prime knot and let $p/q$ be a slope such that 
\begin{equation}\label{eq:JSJ_thm_bound}
\text{$|p/q|> 4g(K)+4$ and $|p|> 6\sqrt{3}(2g(K)-1)$.}
\end{equation}
If the JSJ decomposition of $S_K^3$ is
\[
S_K^3=X_0 \cup X_1 \cup \dots \cup X_n,
\]
where $X_0$ is the outermost piece of this decomposition, then the JSJ decomposition of $S_K^3(p/q)$ is
\begin{equation}\label{eq:general_decomp}
S_K^3(p/q)=X_0(p/q)\cup X_1 \cup \dots \cup X_n.
\end{equation}
\end{thm}
\begin{proof}
It follows from Theorem~\ref{thm:knot_JSJ_decomp} that if $K$ is a prime knot, then $K$ is a torus knot, a non-trivial cable or a knot of hyperbolic type. If $K$ is a torus knot, then the statement follows immediately from Proposition~\ref{prop:JSJ_surgery_torus}. If $K$ is a cable knot, then the desired statement forms a part of Proposition~\ref{prop:JSJ_surgery_cable}. Thus suppose that $K$ is of hyperbolic type.
By Proposition~\ref{prop:length_bound}, the hypothesis that $|p|>6\sqrt{3} (2g(K)-1)$ implies that $\ell_K(p/q)>6$ and so by the 6-theorem \cite{Agol2000BoundsI, Lackenby2003Exceptional}, the outermost piece in the JSJ decomposition of $S_{K}^3$ remains hyperbolic, and hence atoroidal, after the Dehn filling. Thus pieces in the decomposition in \eqref{eq:general_decomp} are all atoroidal or Seifert fibred. Moreover, since $X_0(p/q)$ does not admit any Seifert fibred structure, we see that this decomposition is a minimal such decomposition. That is, \eqref{eq:general_decomp} is the JSJ decomposition of $S_K^3(p/q)$.
\end{proof}

\subsection{Proof of Theorem~\ref{thm:bounded_genus}}
We are now ready for the main result of this section.
\begin{thm}\label{thm:bounded_genus}
\thmboundedgenus{}
\end{thm}

\begin{proof}
\setcounter{claim}{0}
The constant $M$ will be chosen using the JSJ decomposition of the knot complement $S_{K}^3$. Suppose that $S_{K}^3$ has a JSJ decomposition of the form
\[
S_{K}^3=X_0\cup \dots \cup X_n,
\]
where $X_0$ is the outermost piece. We choose constants, $L_0$, $L_1$ and $\beta$ to satisfy the following properties. 
\begin{itemize}
\item If $X_0$ is hyperbolic, then Theorem~\ref{thm:knot_JSJ_decomp} shows that $X_0$ is the complement of a knot in a (possibly trivial) connected sum of solid tori. Thus if $X_0$ is hyperbolic, we can choose $L_0$ to satisfy the conclusions of Theorem~\ref{thm:Lackenby_magic}. If $X_0$ is not hyperbolic, then we simply take $L_0=6$.
\item Applying Lemma~\ref{lem:fixed_JSJ_piece} to each of the $X_i$ for $i=1,\dots, n$, shows that we may pick $L_1\geq 6$ such that none of $X_1, \dots, X_n$ can be obtained by Dehn filling a hyperbolic manifold along a slope of length greater than $L_1$.
\item We choose $\beta\geq 0$ such that for any Seifert fibred $X_i$ all exceptional fibres in $X_i$ have order at most $\beta$.
\end{itemize}
We will show that the constant
 \[M(K)=\max \{\sqrt{3} L_0 (2g(K)-1),\sqrt{3} L_1 (2g(K)-1), 4g(K)+ 4+\beta \}\]
 satisfies the conclusions of the theorem.

Thus suppose that $K'$ is a prime knot with $g(K')\leq g(K)$ such that $S_{K}^3(p/q) \cong S_{K'}^3(p/q)$ for some slope $p/q\in \Q$ satisfying $|p/q|>M(K)$. We will show that $K'$ is isotopic to $K$.

In particular, the slope $p/q$ satisfies the following bounds:
\begin{enumerate}[(i)]
\item\label{it:p_6_bound} $|p|> 6\sqrt{3}(2g(K)-1)$,
\item\label{it:p_L_bound} $|p|> \sqrt{3} \max\{L_0,L_1\}(2g(K)-1)$ and
\item\label{it:pq_fibre_bound} $|p/q|> 4g(K)+ 4+\beta$.
\end{enumerate}
Let
\[
S_{K'}^3=X'_0\cup \dots \cup X'_{n'},
\]
be the JSJ decomposition for $S_{K'}^3$ where $X_0'$ is the outermost piece.
Since $g(K)\geq g(K')$, bounds \eqref{it:p_6_bound} and \eqref{it:pq_fibre_bound} show that \eqref{eq:JSJ_thm_bound} of Theorem~\ref{thm:surgery_JSJ} is satisfied for both $K$ and $K'$.

Thus the JSJ decompositions of $S_{K}^3(p/q)$ and $S_{K'}^3(p/q)$ are given by 
\[
S_{K}^3(p/q)=X_0(p/q)\cup X_1\cup \dots \cup X_{n}
\]
and
\[
S_{K'}^3(p/q)=X'_0(p/q)\cup X'_1\cup \dots \cup X'_{n'},
\]
respectively. Furthermore, notice that $X_0(p/q)$ is a hyperbolic manifold if and only if $K$ is of hyperbolic type and that $X_0(p/q)$ is a Seifert fibred space otherwise. Likewise, $X_0'(p/q)$ is a hyperbolic manifold if and only if $K'$ is of hyperbolic type.
\begin{claim}\label{claim:JSJ_thm1}
The JSJ piece $X'_0(p/q)$ is not homeomorphic to $X_i$ for any $i=1,\dots, n$.
\end{claim}
\begin{proof}[Proof of Claim]
If $K'$ is of hyperbolic type, then the bound \eqref{it:p_L_bound} and Proposition~\ref{prop:length_bound} imply that $\ell_{K'}(p/q)> L_1$. By choice of $L_1$ we see that $X_0'(p/q)$ is not homeomorphic to any of the $X_i$. If $K'$ is a torus knot, then $X_0'(p/q)$ is a closed manifold so certainly not homeomorphic to any of the $X_i$, which all have non-empty boundary. If $K'$ is a cable knot, then \eqref{it:pq_fibre_bound} implies, via Proposition~\ref{prop:JSJ_surgery_cable} that $X_0'(p/q)$ is a Seifert fibred space over the disk with two exceptional fibres, one of which has order at least $\beta+10$. However, the Seifert fibred $X_i$ contain only exceptional fibres of order at most $\beta$.
\end{proof}

Let 
\[
f:S_{K'}^3(p/q) \rightarrow S_{K}^3(p/q)
\]
be an orientation-preserving homeomorphism. Since the JSJ decomposition of a 3-manifold is unique up to isotopy, we may assume that $f$ restricts to give a homeomorphism between $X'_{0}(p/q)$ and a JSJ piece of $S_{K}^3(p/q)$. Given Claim~\ref{claim:JSJ_thm1}, this means that $f$ restricts to give a homeomorphism between $X_0'(p/q)$ and $X_{0}(p/q)$. 

If $K$ is of hyperbolic type, then we see that $K'$ must also be of hyperbolic type. Together \eqref{it:p_L_bound} and Proposition~\ref{prop:length_bound} imply that $\ell_{K'}(p/q)> L_0$. However $L_0$ is a constant for $X_0$ satisfying the conclusions of Theorem~\ref{thm:Lackenby_magic} and so we see that there is a homeomorphism $f': X_0' \rightarrow X_0$ which agrees with $f$ on $\partial X_0'(p/q)$. Thus we can define a homeomorphism $F: S_{K'}^3\rightarrow S_{K}^3$ by
\[
F(x)=\begin{cases}
f(x) &x\in X_1' \cup \dots \cup X_n'\\
f'(x) &x\in X_0'.
\end{cases}
\]

The knot complement theorem of Gordon and Luecke shows that $K$ and $K'$ are isotopic \cite{Gordon1989complements}.

If $X_0$ is a Seifert fibred space (i.e if $K$ is a cable knot or a torus knot), then Proposition~\ref{prop:JSJ_surgery_torus} and Proposition~\ref{prop:JSJ_surgery_cable} combined with the bound  \eqref{it:pq_fibre_bound} imply that $X_0(p/q)$ admits a Seifert fibred structure where the core of the filling torus is an exceptional fibre of order at least $\beta +2$. This implies that $X_0'$ must also be Seifert fibred and, by similar logic, that $X_0'(p/q)$ is Seifert fibred with the core of the filling torus being an exceptional fibre of order at least $\beta +2$. However, the Seifert fibred structure on $X_0(p/q)$ is unique up to isotopy and contains a unique exceptional fibre of order greater than $\beta$. Thus we may isotope $f$ so that it carries the core of the filling torus in $X_0'(p/q)$ to the core of the filling torus in $X_0(p/q)$. Thus $f$ restricts to a homeomorphism between the knot complements $S_{K'}^3$ and $S_{K}^3$. Again the knot complement theorem implies that $K$ and $K'$ are isotopic \cite{Gordon1989complements}. This completes the proof.

\end{proof}

\section{Knot Floer homology and Property~\spliff{}}\label{sec:spaff}
In this section, we discuss the knot Floer homology necessary to define Property~\spliff{} and to explore its properties. We assume familiarity with the basic features of knot Floer homology and Heegaard Floer homology. For convenience we will work throughout with coefficients in $\F=\Z/2\Z$. We begin with some notation and conventions regarding $\Z$-graded $\F[U]$-modules. The variable $U$ will always be assumed to have degree $-2$. We will use $\F_{(d)}$ to denote the graded $\F[U]$-module comprising a copy of the field $\F$ in grading $d$ and $U$-action satisfying $U\F_{(d)}=0$. A $\Z$-graded $\F[U]$-module $M$ \emph{admits an $\F$-factor} if it decomposes as a direct sum of graded $\F[U]$-modules $M\cong M'\oplus \F_{(d)}$ for some $d\in \Z$. 

\begin{defn}\label{def:spaff2}
Let $M$ be a $\Z$-graded $\F[U]$-module. We say that $M$ has \emph{split $\F$-factors} (Property~\spliff{}) if it decomposes as a direct sum of $\F[U]$-modules 
\[M\cong M'\oplus \F^{n_1}_{(d_1)} \oplus \F^{n_2}_{(d_2)},\]
for integers $n_1,n_2\geq 0$ and gradings $d_1, d_2$ where $d_1$ is odd and $d_2$ is even and the module $M'$ does not admit any $\F$-factors.
\end{defn}

A module $M$ admits an $\F$-factor if and only if it contains an element $x\in M$ such that $Ux=0$ and $x\neq Uy$ for all $y\in M$. Thus the condition that the module $M$ has Property~\spliff{} can be alternatively formulated as saying that we have an isomorphism of graded $\F[U]$-modules
\[
\frac{\ker U}{\ker U \cap \mathrm{im}\, U}\cong \F^{n_1}_{(d_1)} \oplus \F^{n_2}_{(d_2)}
\]
where $n_1,n_2\geq 0$, $d_1$ is odd, $d_2$ is even and we are considering $U$ as a map $U:M\rightarrow M$.

We will use $\tp_{(d)}$ to denote the $\F[U]$-module
 \[\tp_{(d)}=\F[U,U^{-1}]/U\F[U]\]
 where the element $1\in \F[U^{-1}]$ is in grading $d\in \Q$. Note that $\tp_{(d)}$ has Property~\spliff{}. When the absolute grading is not relevant will omit the subscript and simply denote this module by $\tp$.

\subsection{Knot Floer homology}
Recall that associated to a knot $K$ in $S^3$, there is the knot Floer complex $CFK^\infty(K)$ which takes the form of a bifiltered chain complex
\[CFK^\infty(K)=\bigoplus_{i,j\in \Z}C\{(i,j)\},\]
whose filtered chain homotopy type is an invariant of $K$. Furthermore, after a suitable chain homotopy, one can assume
\begin{equation}\label{eq:reduced_complex}
C\{(i,j)\}\cong \hfkhat_{*-2i}(K,j-i)
\end{equation}
for all $i,j$ \cite[Reduction Lemma]{Hedden2018botany}.
There is also a natural chain complex isomorphism
\[U: CFK^\infty(K)\longrightarrow CFK^\infty(K),\]
which acts by translating $C\{(i,j)\}$ to $C\{(i-1,j-1)\}$ and lowering the homological grading by 2. This gives $CFK^\infty(K)$ the structure of a finitely-generated $\F[U,U^{-1}]$-module.

For each $k\in \Z$, we take $A_k^+$ to be the quotient complex 
\[A_k^+=C\{i\geq 0 \text{ or } j\geq k\}.\]
We obtain the $\Z$-graded $\F[U]$-module $\Abold_k$ by taking homology $\Abold_k=H_*(A_k^+)$. This module admits a decomposition
\[
\Abold_k\cong \AT_k \oplus \ared{k},
\]
where $\AT_k$ is the submodule $\AT_k=U^N\Abold_k$ for all sufficiently large $N$ and $\ared{k}\cong \Abold_k/\AT_k$.

Let  $B^+$ denote the quotient complex
\[B^+=C\{i\geq 0\}\]
and let $\Bbold=H_*(B^+)$ denote its homology. Since we are working in $S^3$ we have
\[\Bbold\cong \hfp(S^3)\cong \tp_{(0)}.\]
These complexes admit chain maps
\[v_k,h_k \colon A_k^+ \longrightarrow B^+,\]
where $v_k$ is the obvious vertical projection and $h_k$ consists of the composition of a horizontal projection onto $C\{j\geq k\}$, multiplication by $U^k$ and a chain homotopy equivalence. We will use $\vbold_k$ and $\hbold_k$ to denote the maps induced on homology by $v_k$ and $h_k$, respectively. Both $\vbold_k$ and $\hbold_k$ are surjective on homology. Note that $\vbold_k$ preserves the homological grading and that $\hbold_k$ sends an element of grading $x$ to an element of grading $x-2k$ in $\Bbold$. Since the map $\vbold_k$ is surjective, we see that for each $k$ there is an integer $V_k\geq 0$ such that $\AT_k\cong \tp_{(-2V_k)}$. It is not hard to see that the $V_k$ are precisely the sequence of integers defined by Ni and Wu \cite{Ni2015cosmetic}. This sequence is eventually zero and $\nu^+(K)$ is defined to be the minimal $k$ such that $V_k=0$. An important property of $\nu^+$ is that it forms a lower bound for the smooth ball genus \cite{Hom2016refinement}. Consequently, we have that $\nu^+(K)\leq g(K)$, an inequality that we will use frequently throughout this section.

Definition~\ref{def:spaff1} says that a knot $K$ has Property~\spliff{} if and only if the graded $\F[U]$-modules $\Abold_k$ all have Property~\spliff{}. We will make use of the following reformulation of Property~\spliff{} for knots.

\begin{prop}
Let $K$ be a knot in $S^3$. Then $K$ has Property~\spliff{} if and only if $\ared{k}$ has Property~\spliff{} for all $k\geq 0$.
\end{prop}
\begin{proof}
The module $\Abold_k$ admits a decomposition
\[
\Abold_k \cong \AT_k \oplus \ared{k}
\] where $\AT_k$, being isomorphic to $\tp$, does not contain any $\F$-summands. Thus the module $\Abold_k$ has Property~\spliff{} if and only if $\ared{k}$ has Property~\spliff{}. One implication of the proposition follows immediately. Conversely suppose that $\ared{k}$ has Property~\spliff{} for all $k\geq 0$.  Since $\ared{k}$ is isomorphic to $\ared{-k}$ up to an overall grading shift\footnote{The precise grading shift is not relevant for our purposes, but can be easily computed. For $k\geq 0$ the isomorphism $\ared{k}\rightarrow\ared{-k}$ lowers the absolute grading by $2k$.} \cite[Lemma~2.3]{Hom2015reducible}, this implies implies that $\ared{k}$ and hence $\Abold_k$ has Property~\spliff{} for all $k$.
\end{proof}
In fact, when it comes to verifying Property~\spliff{}, it is necessary only to consider $\ared{k}$ for $k$ in the range $0\leq k\leq g(K)-1$.

\begin{prop}\label{prop:ared_properties}
Let $K$ be a knot in $S^3$ with genus $g=g(K)$.
\begin{enumerate}[(i)]
\item\label{it:ared=0} For all $|k|\geq g$, we have $\ared{k}=0$.
\item\label{it:ared_g-1_calc} If $\nu^+(K)< g(K)$ then $\ared{g-1}$ and $\ared{1-g}$ are non-trivial and their $\F[U]$-module structure is the one satisfying $U  \ared{g-1} =U  \ared{1-g}=0$. Furthermore, as a graded $\F$-vector space, there is an isomorphism $\ared{g-1}\cong \hfkhat_{*+2}(K,g)$.

\end{enumerate}
\end{prop}
\begin{proof}
We establish the results for $k\geq g$ and $\ared{g-1}$. This is sufficient to establish the results for $k\leq -g$ and $\ared{1-g}$, since up to an overall grading shift there is an isomorphism $\ared{k}\cong\ared{-k}$ for all $k$ \cite[Lemma~2.3]{Hom2015reducible}. First note that if $k\geq g$, then the complexes $A^+_k$ and $B^+$ are in fact equal. Thus we have that $\Abold_k\cong \hfp(S^3)\cong \tp_{(0)}$. This implies that $\ared{k}=0$ for $k\geq g$.
 
Let $D$ be the subquotient complex of $CFK^\infty$ given by $C\{i\leq -1\,\text{and}\, j\geq g-1\}$. We have the short exact sequence of chain complexes which is invariant under the $U$-action:
\[0\longrightarrow D \longrightarrow A^+_{g-1} \overset{\vbold_{g-1}}\longrightarrow B^+ \rightarrow 0.
\]
As $\vbold_{g-1}$ is surjective on homology, the exact triangle induced by this sequence gives a short exact sequence of $\F[U]$-modules
\[0\longrightarrow H_*(D) \longrightarrow \Abold_{g-1} \overset{\vbold_{g-1}}\longrightarrow \Bbold \rightarrow 0.
\]
Thus we see that $H_*(D)$ is isomorphic to $\ker \vbold_{g-1}$. Since $\nu^+(K)<g$ we have that $V_{g-1}=0$ and so $\ker \vbold_{k}$ is isomorphic to $\ared{g-1}$.

However the complex $D$ is supported in a single bigrading:
\[
D=C\{(-1,g-1)\}\cong \hfkhat_{*+2}(K,g).
\]
with trivial differential. Thus $\ared{g-1}\cong H_*(D)\cong \hfkhat_{*+2}(K,g)$ and $U\ared{g-1}=0$. Since knot Floer homology detects the genus \cite{Ghiggini2008fibredness, Ni2007fibredness}, we have that $\hfkhat(K,g)$ is non-trivial.
\end{proof}
We can use Proposition~\ref{prop:ared_properties} to exhibit knots which do not have Property~\spliff{}. By Proposition~\ref{prop:ared_properties}\eqref{it:ared_g-1_calc}, any knot for which $\nu^+(K)<g(K)$ and $\hfkhat(K,g)$ is supported in at least three distinct gradings does not have Property~\spliff{}.
\begin{exam}
The knot $K=11n34\# 11n34$ does not satisfy Property~\spliff{}. The knot $11n34$ has genus $g(11n34)=3$ and
\[
\hfkhat(11n34, 3)\cong \F_{(3)}\oplus \F_{(4)}.
\]
By additivity of the Seifert genus and the K{\"u}nneth formula for knot Floer homology, we have $g(K)=6$ and
\[
\hfkhat(K, 6)\cong \F_{(6)}\oplus \F_{(7)}^2 \oplus \F_{(8)}.
\]
Since
\[\nu^+(K)\leq g_4(K)\leq 2g_4(11n34)\leq 2<g(K)=6,\]
one may apply Proposition~\ref{prop:ared_properties}\eqref{it:ared_g-1_calc} to obtain
\[\ared{5}(K)\cong \F_{(4)}\oplus \F_{(5)}^2 \oplus \F_{(6)},\]
showing that $K$ does not have Property~\spliff{}. 
\end{exam}

\begin{rem}
Using further properties of the $V_k$, one can generalize the calculation of $\ared{g-1}$ in Proposition~\ref{prop:ared_properties}\eqref{it:ared_g-1_calc} to calculate $\ared{g-1}$ when $\nu^+(K)=g(K)$. When $\nu^+(K)=g(K)$, the module structure on $\ared{g-1}$ again satisfies $U \ared{g-1}=0$, but as a graded vector space the structure on $\ared{g-1}$ satisfies
 \[\F_{(-2)}\oplus \ared{g-1}\cong \hfkhat_{*+2}(K,g).\]
\end{rem}

\subsection{The mapping cone formula}
In order to describe the Heegaard Floer homology of $S^3_{K}(p/q)$, one needs a way to label its \spinc-structures. This labelling takes the form of an affine bijection defined in terms of relative \spinc-structures on $S^3\setminus \nu K$ \cite{Ozsvath2011rationalsurgery}:
\begin{equation}\label{eq:spinccorrespondence}
\phi_{K,p/q} \colon \Zp \longrightarrow \spinc(S_{K}^3(p/q)).
\end{equation}

In general, if we have two surgery descriptions for a manifold
\[Y\cong S_{K}^3(p/q)\cong S_{K'}^3(p/q)\cong Y,\]
then $\phi_{K,p/q}$ and $\phi_{K',p/q}$ will result in different labellings on $\spinc(Y)$. However, by using the $d$-invariants one can show that for $p$ sufficiently large relative to $q$ these labellings will be the same up to conjugation. The following lemma gives a precise quantification of what sufficiently large means in this context.

\begin{lem}[Lemma~3.3, \cite{McCoy2020torus_char}]\label{lem:labellingbound}
Let $K$ and $K'$ be knots in $S^3$ such that $S_{K}^3(p/q)\cong S_{K'}^3(p/q)$ for some $p/q\in \Q$. If $p,q>0$ satisfy
\[p\geq 4q\nu^+(K)+4q^2-2q+12,\]
then up to conjugation of \spinc-structures the labelling maps satisfy $\phi_{K,p/q}=\phi_{K',p/q}$. \qed
\end{lem}
Now we describe how $CFK^\infty(K)$ determines $\hfp(S_{K}^3(p/q))$. Consider the groups
\[
\Ap_i=\bigoplus_{s\in \Z}\left(s,\Abold_{\lfloor\frac{ps+i}{q}\rfloor}\right) \quad \text{and} \quad \Bp_i=\bigoplus_{s\in \Z}\left(s,\Bbold\right),
\]
and the maps
\[\vbold_{\lfloor\frac{ps+i}{q}\rfloor}\colon \left(s,\Abold_{\lfloor\frac{ps+i}{q}\rfloor}\right) \rightarrow \left(s,\Bbold \right) \quad \text{and} \quad
\hbold_{\lfloor\frac{ps+i}{q}\rfloor}\colon \left(s,\Abold_{\lfloor\frac{ps+i}{q}\rfloor}\right) \rightarrow \left(s+1,\Bbold \right),\]
where $\vbold_k$ and $\hbold_k$ are the maps on homology induced by $v_k$ and $h_k$ as in the previous section. These maps can be added together to obtain a chain map
\[\Dbold\colon \Ap_i \rightarrow \Bp_i,\]
defined by
\begin{equation}\label{eq:Dbold_def}
\Dbold(s,x)= \left(s,\vbold_{\lfloor\frac{ps+i}{q}\rfloor}(x)\right)+ \left(s+1,\hbold_{\lfloor\frac{ps+i}{q}\rfloor}(x)\right).
\end{equation}
The module $\hfp(S_{K}^3(p/q),i)$ is computed in terms of the mapping cone on $\Dbold$.
\begin{thm}[Ozsv{\'a}th-Szab{\'o}, \cite{Ozsvath2011rationalsurgery}]\label{thm:mappingcone}
For any knot $K$ in $S^3$, let $\Xipq$ be the mapping cone of $\Dbold$. Then there is a isomorphism of $\F[U]$-modules
\[H_*(\Xipq) \cong \hfp(S_{K}^3(p/q),i).\]\qed
\end{thm}

Furthermore, as discussed in \cite[Section~7.2]{Ozsvath2011rationalsurgery}, $\Xipq$ admits a $\Q$-grading that induces the absolute $\Q$-grading on $\hfp(S_{K}^3(p/q),i)$. Explicitly, the grading on $\Bp_i$ is independent of the knot $K$ and is determined by surgeries on the unknot. The grading on $\Ap_i$ satisfies the property that $\Dbold$ decreases the grading by one. For $p/q>0$ and $0\leq i \leq p-1$, this implies that gradings are as follows \cite{Ni2015cosmetic}. For each $s\in \Z$ let $m_{i,s}\in \Q$ denote the minimal grading on the tower $(s, \Bbold)\subseteq \Bp_i$. These satisfy \cite[Section~3.3]{Ni2014characterizing}
\begin{equation}\label{eq:mingrading}
\text{$m_{i,0}=d(S_{U}^3(p/q),i)-1$ and $m_{i,s+1}=m_{i,s}+2\left\lfloor\frac{i+ps}{q}\right\rfloor$ for any $s\in\Z$.}
\end{equation}

The second formula comes from combining \eqref{eq:Dbold_def} with the fact that $\vbold_k:\Abold_k \rightarrow \Bbold$ preserves homological grading and $\hbold_k:\Abold_k \rightarrow \Bbold$ shifts the homological grading by $2k$.
Thus if we define integers $D_{i,s}$ by the formula
\begin{equation*}
D_{i,s}=\begin{cases} \sum_{k=0}^{s-1} 2\floorfrac{i+pk}{q} &\text{if $s> 0$}\\
0 & \text{if $s= 0$}\\
\sum_{k=1}^{-s} 2\ceilfrac{kp-i}{q} &\text{if $s<0$,}
\end{cases}
\end{equation*}
then the $m_{i,s}$ satisfy
\begin{equation*}
m_{i,s}=d(S_{U}^3(p/q),i)-1 + D_{i,s}.
\end{equation*}

Thus we see that the $\Q$-grading on $(s,\Abold_{\lfloor\frac{ps+i}{q}\rfloor})\subseteq \Ap_i$ is determined by the homological grading on $\Abold_{\lfloor\frac{ps+i}{q}\rfloor}$ but shifted by a constant of the form:
\begin{equation}\label{eq:totalgradshift}
d(S_{U}^3(p/q),i)+D_{i,s}.
\end{equation}

Using these absolute gradings, Ni and Wu showed that for any $p/q>0$ and any $0\leq i \leq p-1$, the $d$-invariants of $S_{K}^3(p/q)$ can be calculated by \cite[Proposition~1.6]{Ni2015cosmetic}
\begin{equation}\label{eq:NiWuformula}
d(S_{K}^3(p/q),i)=d(S_{U}^3(p/q),i)-2\max\left\{V_{\lfloor\frac{i}{q}\rfloor},V_{\lceil\frac{p-i}{q}\rceil}\right\}.
\end{equation}

Together Lemma~\ref{lem:labellingbound} and equation \eqref{eq:NiWuformula} shows that the $V_k$ can be recovered from sufficiently large surgeries on a knot.
\begin{prop}\label{prop:Vi_recovery}
Let $K$ and $K'$ be knots in $S^3$ such that $S_{K}^3(p/q)\cong S_{K'}^3(p/q)$ for some $p/q\in \Q$. If $p,q>0$ satisfy
\[p\geq 4q\nu^+(K)+4q^2-2q+12.\]
then $\nu^+(K)=\nu^+(K')$ and $V_k(K)=V_k(K')$ for all $k\geq 0$.
\end{prop}
\begin{proof}
Lemma~\ref{lem:labellingbound} allows us to assume that the labelling maps on 
\[\spinc(S_{K}^3(p/q))=\spinc( S_{K'}^3(p/q))\]
agree.
This implies that for all $i$ in the range $0\leq i \leq p-1$ we have an isomorphism
\[\hfp(S_{K}^3(p/q),i)\cong\hfp(S_{K'}^3(p/q),i).\]
By comparing the $d$-invariants of these groups and applying \eqref{eq:NiWuformula}, this shows that $V_k(K)=V_k(K')$ for all $0\leq k \leq \lfloor \frac{p+q-1}{2q} \rfloor$. Since $p/q>2\nu^+(K)-1$, it follows that $V_{\lfloor \frac{p+q-1}{2q} \rfloor}(K)=V_{\lfloor \frac{p+q-1}{2q} \rfloor}(K')=0$. This shows that $V_k(K)=V_k(K')$ for all $k\geq 0$. The equality $\nu^+(K)=\nu^+(K')$ follows immediately.
\end{proof}

One can also compute the reduced Heegaard Floer homology groups. We require only the special case when $p/q\geq 2\nu^+(K)-1$. The following proposition can easily be derived from \cite[Corollary~12]{Gainullin2017mapping} or \cite[Proposition~3.6]{Ni2014characterizing}.
\begin{prop}\label{prop:redHFp}
If $p/q\geq 2\nu^+(K)-1$ and $p/q>0$, then as $\Q$-graded $\F[U]$-modules, we have
\begin{equation}\label{eq:redHFp_formula}
\hfred(S_{K}^3(p/q),i)\cong \bigoplus_{s\in\Z} \ared{\lfloor\frac{i+ps}{q}\rfloor}[d(S_{U}^3(p/q),i)+D_{i,s}].
\end{equation}
\qed
\end{prop}
Here we are using $\ared{k}[r]$ to denote a copy of $\ared{k}$ but with its grading shifted by $r\in \Q$. That is, if an element $\xi\in \ared{k}$ has grading $g$, then $\xi$ has grading $g+r$ in $\ared{k}[r]$. Thus, in accordance with the discussion above, Proposition~\ref{prop:redHFp} simply says that each copy of $\ared{\lfloor\frac{i+ps}{q}\rfloor}$ in the sum \eqref{eq:redHFp_formula} comes with the grading inherited from the $\Q$-grading on the summand $(s,\Abold)\subset \Xipq$.

We note that when $p/q\geq 2g(K)-1$, \eqref{eq:redHFp_formula} can be simplified.
\begin{lem}\label{lem:redHFp_small_g}
Let $K$ be a non-trivial knot. If $p/q\geq 2g(K)-1$, then as $\Q$-graded $\F[U]$-modules, we have
\begin{equation}\label{eq:redHFp_formula_small_g}
\hfred(S_{K}^3(p/q),i)
\cong \begin{cases}
\ared{\floorfrac{i}{q}}[d(S_{U}^3(p/q),i)] & 0\leq i <qg(K)\\
0 &qg(K)\leq i < p+q -qg(K)\\
\ared{\floorfrac{i-p}{q}} [d(S_{U}^3(p/q),i)+D_{i,-1}] & p+q -qg(K)\leq i \leq p-1.
\end{cases}
\end{equation}
\end{lem}
\begin{proof}
Proposition~\ref{prop:ared_properties}\eqref{it:ared=0} shows that $\ared{k}=0$ for $|k|\geq g(K)$. Thus to derive \eqref{eq:redHFp_formula_small_g} from \eqref{eq:redHFp_formula} we need to determine, for a given $i$, the values $s\in\Z$ such that
\[
1-g(K)\leq \floorfrac{i+ps}{q}\leq g(K)-1.
\]
Equivalently, we need to find the $s\in \Z$ such that
\begin{equation}\label{eq:small_g_analysis}
1-g(K)\leq \frac{i+ps}{q}\leq g(K)-\frac{1}{q}.
\end{equation}
Since $i\geq 0$ and $p/q\geq 2g(K)-1\geq g(K)$, the upper bound in \eqref{eq:small_g_analysis} can be satisfied only if $s\leq 0$. Similarly, since $i\leq p-1$, the lower bound in \eqref{eq:small_g_analysis} can be satisfied only if $s\geq -1$. Analysing the possibilities $s=0$ and $s=-1$,  shows that \eqref{eq:redHFp_formula} reduces to \eqref{eq:redHFp_formula_small_g}.
\end{proof}

The following calculation shows that certain pairs of $D_{i,k}$ have differ by a non-zero even integer. This plays a key role in the proof of Lemma~\ref{lem:ared_recovery} below, where it is used to calculate the difference in absolute gradings between certain terms appearing in the mapping cone formula.
\begin{lem}\label{lem:difference_calc}
If $0\leq i<p$ is such that $\floorfrac{i}{q}=\floorfrac{i+1}{q}$, then for any $t\in\Z$, 
\[D_{i+1,tq}-D_{i,tq}=2t.\]
\end{lem}
\begin{proof}
Observe that
\[
\floorfrac{i+1+pk}{q} - \floorfrac{i+pk}{q}=\begin{cases}
1 &\text{if $pk \equiv -i-1 \bmod{q}$}\\
0&\text{otherwise.}
\end{cases}
\]
Thus, if $t\geq 0$, then \eqref{eq:totalgradshift} gives
 \begin{align*}
 D_{i+1,tq}-D_{i,tq} &= 2\sum_{k=0}^{tq-1} \floorfrac{i+1+pk}{q} - \floorfrac{i+pk}{q}\\
 &=2\#\{ 0\leq k\leq tq-1 \mid pk \equiv-i -1 \bmod q \}\\
 &=2t.
 \end{align*}
 Similarly if $t<0$ we have that
  \begin{align*}
 D_{i+1,tq}-D_{i,tq} &=2\sum_{k=1}^{-tq-1} \ceilfrac{kp-i-1}{q} - \ceilfrac{kp-i}{q}\\
 &=-2\#\{ 1\leq k\leq |tq|-1 \mid pk \equiv i+ 1 \bmod q \}\\
 &=2t.
 \end{align*}
 \end{proof}
Using Lemma~\ref{lem:labellingbound} and Proposition~\ref{prop:redHFp} we see that under some circumstances the $\ared{k}$ can also be recovered from sufficiently large surgeries on a knot. This is the crucial technical step in the proof of Theorem~\ref{thm:HFKbound}.
\begin{lem}\label{lem:ared_recovery}
Let $K$ and $K'$ be knots in $S^3$ such that $S_{K}^3(p/q)\cong S_{K'}^3(p/q)$ for some $p/q\in \Q$ satisfying 
\[
p\geq 12+4q^2 -2q +4qg(K).
\]
If one of the following conditions holds:
\begin{enumerate}[(i)]
\item\label{it:q_geq_3} $q\geq 3$ or
\item\label{it:q=2} $q=2$ and $K$ has Property~\spliff{},
\end{enumerate}
then $\ared{k}(K)\cong\ared{k}(K')$ for all $k\geq 0$.
\end{lem}
Before we embark on the full proof of Lemma~\ref{lem:ared_recovery} we will summarize the main ideas. Let $Y$ be the 3-manifold such that $Y\cong S_K^3(K)\cong S_K^3(K)$. Let $\ared{k}:=\ared{k}(K)$ and $\aredp{k}:=\ared{k}(K')$ denote the modules $\ared{k}$ derived from the knot Floer complexes of $K$ and $K'$, respectively. Using the mapping cone formula as in Proposition~\ref{prop:redHFp} and Lemma~\ref{lem:redHFp_small_g} allows us to derive expressions for $\hfred(Y)$ in terms of both the $\ared{k}$ and the $\aredp{k}$. A careful comparison of these two descriptions for $\hfred(Y)$ gives relations between the $\ared{k}$ and $\aredp{k}$. For all $i$ in the range $0\leq i < qg(K)$, we have
\begin{equation}\label{eq:technical_sketch}
\ared{\floorfrac{i}{q}}\cong\bigoplus_{t\in \Z}  \aredp{tp+\lfloor \frac{i}{q} \rfloor}[D_{i,tq}].
\end{equation}
For a fixed $k$, we have $\floorfrac{i}{q}=k$ for all $i=qk, \dots, qk+q-1$. So when $q\geq 2$, \eqref{eq:technical_sketch} actually yields multiple expressions for the same $\ared{k}$ differing only in the grading shifts for the $\bm{A}'_{{\rm red}}$ terms occurring on the right hand side. However, the differences in these grading shifts are precisely those studied in Lemma~\ref{lem:difference_calc}. In particular, if there is some $t\neq 0$ and some $0\leq j<g(K)$ such that $\aredp{pt+j}$ contains an $\F$-factor, then we will see that this $\F$-factor will necessarily show up $q$ times in $\ared{j}$ with $q$ distinct gradings all of the same parity. Notably the existence of such an $\F$-factor implies that $K$ does not have Property~\spliff{}. 

Now suppose that there is some $k$ such that $\ared{k}$ and $\aredp{k}$ are not isomorphic. By analysing the non-zero terms in the right hand side of \eqref{eq:technical_sketch}, we deduce that of $g(K')-1$ can be written in the form $pt+j$ for some $t\neq 0$ and $0\leq |j|<g(K)$. On the other hand Proposition~\ref{prop:ared_properties}\eqref{it:ared_g-1_calc} implies that $\aredp{g(K')-1}$ must contain an $\F$-factor. If $q\geq 2$, then we can use one of $\aredp{g(K')-1}$ or  $\aredp{1-g(K')}$ as in the previous paragraph to conclude that $K$ does not have Property~\spliff{}.

If $q\geq 3$, then we obtain a contradiction. If we fix an $i$ such that $k=\floorfrac{i}{q}=\floorfrac{i+1}{q}=\floorfrac{i+2}{q}$, then \eqref{eq:technical_sketch} yields three expressions for the group $\ared{k}$ each with the same summands, but differing in absolute grading. These expressions are not mutually compatible if $\ared{k}$ and $\aredp{k}$ are not isomorphic, since an application of Lemma~\ref{lem:difference_calc} shows that the gradings in the expression from $\floorfrac{i+1}{q}$ will be smaller than those appearing in the expressions from $\floorfrac{i}{q}$ and $\floorfrac{i+2}{q}$.

\begin{proof}[Proof of Lemma~\ref{lem:ared_recovery}]
\setcounter{claim}{0}
We may assume that $K$ is non-trivial -- if $K$ were trivial, then we already have the stronger conclusion that $K'$ is also trivial, since every slope is characterizing for the unknot. 
By Proposition~\ref{prop:Vi_recovery}, we have that $\nu^+(K)=\nu^+(K')$. Thus we have 
\begin{equation}\label{eq:pq_bound}
p/q>2g(K)\geq 2\nu^+(K)=2\nu^+(K').
\end{equation}
Let $Y$ denote the 3-manifold $Y\cong S_{K}^3(p/q)\cong S_{K'}^3(p/q)$.
By Lemma~\ref{lem:labellingbound}, we may assume that the labelling maps $\phi_{K,p/q}, \phi_{K',p/q}: \Zp \rightarrow \spinc(Y)$ coincide.
For $i=0, \dots, p-1$, this allows us to define $\hfp(Y,i)$ by
\[\hfp(Y,i)\cong \hfp(S_{K}^3(p/q),\phi_{K,p/q}(i))\cong\hfp(S_{K'}^3(p/q),\phi_{K,p/q}(i)).\]
In order to simplify the comparison between gradings in different groups, we will normalize the $\Q$-grading on $\hfp(Y,i)$ by subtracting $d(S_U^3(p/q),i)$. With this normalization in place, \eqref{eq:NiWuformula} shows that the absolute $\Q$-grading on $\hfp(Y,i)$ becomes a $\Z$-grading.

We will use $\ared{k}:=\ared{k}(K)$ and $\aredp{k}:=\ared{k}(K')$ to denote the modules $\ared{k}$ derived from the knot Floer complexes of $K$ and $K'$ respectively.

Since $p/q>2g(K)$, Lemma~\ref{lem:redHFp_small_g} applies to calculate $\hfred(Y)$ in terms of the $\ared{k}$. Taking into account our normalized grading, this gives

\begin{equation}\label{eq:mapping_cone_app_1}
\hfred(Y,i)\cong \begin{cases}
\ared{\floorfrac{i}{q}} & 0\leq i <qg(K)\\
0 &qg(K)\leq i < p+q -qg(K)\\
\ared{\floorfrac{i-p}{q}} [D_{i,-1}] & p+q -qg(K)\leq i \leq p-1,
\end{cases}
\end{equation}
 for $i$ in the range $0\leq i \leq p-1$.

In particular, for $k$ in the range $0\leq k< g(K)$ we have isomorphisms
\[\hfred(Y,kq)\cong \dotsb \cong \hfred(Y,kq+q-1)\cong \ared{k}.\]

Given \eqref{eq:pq_bound}, we may also apply Proposition~\ref{prop:redHFp} to the surgery description of $Y$ in terms of $K'$. This yields the expression:
\begin{equation}\label{eq:direct_application}
 \hfred(Y,i)\cong\bigoplus_{s\in \Z}  \aredp{\lfloor \frac{ps+i}{q} \rfloor}[D_{i,s}].
\end{equation}
The following claim shows that \eqref{eq:direct_application} can be simplified dramatically.
\begin{claim}
For all $i$ in the range $0\leq i \leq qg(K)$, we have
\begin{equation}\label{eq:modq_decomp}
\hfred(Y,i)\cong\bigoplus_{t\in \Z}  \aredp{tp+\lfloor \frac{i}{q} \rfloor}[D_{i,tq}].
\end{equation}
\end{claim}
Note that equating \eqref{eq:mapping_cone_app_1} and \eqref{eq:modq_decomp} yields the expression \eqref{eq:technical_sketch}  discussed in the proof outline above.
\begin{proof}[Proof of Claim]
We establish \eqref{eq:modq_decomp} by establishing it first for $i=qg(K)$ and then inducting downwards on the index $i$.

For $i= gq$, we obtain from \eqref{eq:mapping_cone_app_1} and \eqref{eq:direct_application}:
\begin{align*}
\bigoplus_{t\in \Z}  \aredp{tp+\floorfrac{i}{q}}[D_{i,tq}]&= \bigoplus_{t\in \Z}  \aredp{\floorfrac{pqt+i}{q}} [D_{i,tq}]\\
& \subseteq \bigoplus_{s\in \Z}\aredp{\lfloor \frac{ps+i}{q} \rfloor}[D_{i,s}]\\
&\cong \hfred(Y,i) =0.
\end{align*}
This gives \eqref{eq:modq_decomp} for $i= gq$.

Now suppose that \eqref{eq:modq_decomp} holds for index $i+1$ in the range $0<i+1\leq qg(K)$. That is,
\begin{equation}\label{eq:i+1_case}
\hfred(Y,i+1)\cong\bigoplus_{t\in \Z}  \aredp{tp+\floorfrac{i+1}{q}}[D_{i+1,tq}].
\end{equation}

First suppose that $i\equiv -1 \bmod q$. For any $s\not\equiv 0\bmod q$, \eqref{eq:i+1_case} implies that
\[
\aredp{\floorfrac{ps+i+1}{q}}=0.
\]

Furthermore, for any $s\not\equiv 0\bmod q$, we have that $\floorfrac{ps+i}{q}=\floorfrac{ps+i+1}{q}$, since $ps+i+1\equiv ps\not\equiv 0 \bmod q$. This yields
\[
\bigoplus_{\substack{s\in \Z\\ s\not\equiv 0\bmod q}} \aredp{\floorfrac{ps+i}{q}}=\bigoplus_{\substack{s\in \Z\\ s\not\equiv 0\bmod q}} \aredp{\floorfrac{ps+i+1}{q}}=0.
\]
Thus the only non-zero terms in the right hand side of \eqref{eq:direct_application} occur when the index $s$ is divisible by $q$. Thus \eqref{eq:direct_application} reduces to \eqref{eq:modq_decomp} when $i\equiv -1 \bmod q$.

Alternatively, suppose that $i\not\equiv -1 \bmod q$. In this case we have that $\hfred(Y,i)\cong\hfred(Y,i+1)$ and  that$\floorfrac{i}{q}=\floorfrac{i+1}{q}$. Let $M$ be the module
\[
M=\bigoplus_{t\in \Z}  \aredp{\lfloor \frac{pqt+i}{q} \rfloor}[D_{i,tq}].
\]
Note that $M$ arises a summand of the right hand side of \eqref{eq:direct_application} and is hence isomorphic to a submodule of $\hfred(Y,i)$. However, calculating the dimension of $M$ as an $\F$-vector space, we obtain
\begin{align*}
\dim_\F M &=\sum_{t\in\Z} \dim_\F \aredp{pt+ \floorfrac{i}{q}}\\
&=\sum_{t\in\Z} \dim_\F \aredp{pt+ \floorfrac{i+1}{q}}\\
&=\dim_\F \hfred(Y,i+1) =\dim_\F \hfred(Y,i).
\end{align*}
 Thus $M$ must actually be isomorphic to $\hfred(Y,i)$, yielding the inductive step in this case too.
 \end{proof}
 Using \eqref{eq:mapping_cone_app_1} and \eqref{eq:modq_decomp} we obtain restrictions on which $\aredp{k}$ can be non-zero. 
\begin{claim}\label{claim:non-trivial}
 If $\aredp{k}$ is non-trivial for some $k$, then there are integers $t$ and $j$ such that $k=pt+j$ and $|j|<g(K)$.
\end{claim}
\begin{proof}[Proof of Claim]
Suppose that $\aredp{k}$ is non-trivial. There are integers $s$ and $i$ such that $qk=ps+i$ and $0\leq i\leq p-1$. According to \eqref{eq:direct_application}, the module $\aredp{k}$ appears as a non-trivial summand in $\hfred(Y,i)$. Thus, by \eqref{eq:mapping_cone_app_1} we must have that $i<qg(K)$ or $i\geq p+q-qg(K)$.

If $i<qg(K)$, then that $\aredp{k}$ arises as a non-trivial summand in \eqref{eq:modq_decomp}. Thus $k$ takes the form $k=tp+j$, where $j=\floorfrac{i}{q}<g(K)$.

On the other hand, if $i\geq p+q-qg(K)$, then we can write $-qk=-p-ps+p-i$, where $p-i\leq gq-q$. By \eqref{eq:direct_application}, $\aredp{-k}$ arises as a summand in $\hfred(Y,p-i)$. Since $\aredp{k}$ and $\aredp{-k}$ are isomorphic up to an overall grading shift \cite[Lemma~2.3]{Hom2015reducible}, $\aredp{-k}$ is also non-trivial. Hence \eqref{eq:modq_decomp} applied to $\hfred(Y,p-i)$ shows that $-k$ takes the form $-k=tp+j$, where $j=\floorfrac{p-i}{q}\leq g(K)-1$. This shows that $k$ can also be written in the required form.
\end{proof}

Next we deduce the existence of an $\F$-factor in $\aredp{g(K')-1}$ when the $\ared{k}$ and $\aredp{k}$ are not all isomorphic. This $\F$-factor will later be used to deduce that $K$ does not have Property~\spliff{}.

\begin{claim}\label{claim:genus}
If there is some $k\geq 0$ such that $\ared{k}$ and $\aredp{k}$ are not isomorphic, then there exist integers $t\neq 0$ and $0\leq j<g(K)$ such that $\aredp{pt+j}$ contains at least one $\F$-factor.
\end{claim}
\begin{proof}[Proof of Claim]
First we show $g(K')>g(K)$. If there exists $k$ such that $|k|\geq g(K)$ and $\aredp{k}$ is non-trivial, then Proposition~\ref{prop:ared_properties}\eqref{it:ared=0} implies that $g(K')>g(K)$. Thus suppose that $\aredp{k}=0$ whenever $|k|\geq g(K)$. However, with such an assumption we find that for all $k$ in the range $0\leq k <g(K)$, \eqref{eq:mapping_cone_app_1} and \eqref{eq:modq_decomp} imply that
\begin{align*}
\ared{k}&\cong \hfred(Y,kq)\\
&\cong \aredp{k}\oplus \bigoplus_{t\in \Z\setminus\{0\}}  \aredp{pt+k}[D_{kq,tq}]\\
&=\aredp{k},
\end{align*}
where we used the fact that $|pt+k|>g(K)$ for all $t\neq 0$ to obtain the last line.

Thus if there is some $k\geq 0$ such that $\ared{k}\not\cong \aredp{k}$, then
\[
g(K')>g(K)\geq \nu^+(K)=\nu^+(K').
\]
Proposition~\ref{prop:ared_properties}\eqref{it:ared_g-1_calc} implies that $\aredp{g(K')-1}$ and $\aredp{1-g(K')}$ are both non-zero and comprise entirely of $\F$-factors. By Claim~\ref{claim:non-trivial}, we may write
\[
g(K')-1= pt'+j'
\]
where $|j'|<g(K)$. Moreover since $g(K')>g(K)$, we see that $t'\geq 1$. If $j'\geq 0$, then we take $t=t'$ and $j=j'$ to get integers satisfying the conclusions of the claim. If $j'<0$, then we take $j=-j'$ and $t=-t'$ to complete the proof of the claim.
\end{proof}

For the remainder of the proof we assume that there is some $k\geq 0$ such that $\ared{k}\not\cong \ared{k}'$. By Claim~\ref{claim:genus} this implies the existence of $t\neq 0$ and $0\leq j <g(K)$ such that $\aredp{pt+j}$ is non-trivial and, in particular, admits an $\F$-factor.
Suppose that $q\geq 2$. We will show that this implies that $K$ does not have Property~\spliff{}.
 Suppose that the $\F$-factor of $\aredp{pt+j}$ is supported in degree $d$. By \eqref{eq:modq_decomp}, we can conclude that the module $\hfred(Y,jq)$ admits an $\F$-factor in degree $d+D_{jq,tq}$ and that $\hfred(Y,jq+1)$ admits an $\F$-factor in degree $d+D_{jq+1,tq}$. However, by \eqref{eq:mapping_cone_app_1}, we have that
\[
\hfred(Y,qj)\cong \ared{j}\cong \hfred(Y,qj+1).
\]
Thus $\ared{j}$ admits $\F$-factors in degrees $d+D_{i,tq}$ and $d+D_{i+1,tq}$. Since $t\neq 0$, Lemma~\ref{lem:difference_calc} shows that these degrees are distinct and of the same parity. Thus $\ared{j}$ and, hence $K$, does not have Property~\spliff{}. This proves condition \eqref{it:q=2} of the lemma.

Now suppose that $q\geq 3$. We will show that this results in a contradiction. By considering \eqref{eq:modq_decomp}, for $i=qj$, $qj+1$ and $qj+2$, we obtain isomorphisms
\[
\ared{j}\cong \bigoplus_{s\in \Z}  \aredp{sp+j}[D_{jq,sq}] \cong \bigoplus_{s\in \Z}  \aredp{sp+ j}[D_{jq+1,sq}] \cong \bigoplus_{s\in \Z}  \aredp{sp+ j}[D_{jq+2,sq}].
\] 
Since $D_{i,0}=0$ for all $i$, the summand corresponding to $s=0$ in each of the above groups is isomorphic to $\aredp{j}[0]$. Ignoring this common summand, implies the existence of isomorphisms
\begin{equation}\label{eq:q_geq_isos}
\bigoplus_{\substack{s\in \Z\\ s\neq 0}}  \aredp{sp+j}[D_{jq,sq}] \cong \bigoplus_{\substack{s\in \Z\\ s\neq 0}}  \aredp{sp+ j}[D_{jq+1,sq}] \cong \bigoplus_{\substack{s\in \Z\\ s\neq 0}}  \aredp{sp+ j}[D_{jq+2,sq}].
\end{equation}
Moreover, since $\aredp{tp+j}$ is non-zero, these groups are all non-trivial. Let $d$ be the maximal grading of an element in $\bigoplus_{\substack{s\in \Z\\ s\neq 0}}  \aredp{sp+ j}[D_{jq+1,sq}]$. Thus, there is an integer $r\neq 0$ such that $\aredp{rp+ j}[D_{jq+1,rq}]$ contains an element of grading $d$. If $r>0$, then Lemma~\ref{lem:difference_calc} implies that $\aredp{rp+ j}[D_{jq+2,rq}]$ contains an element of grading
\[
d-D_{jq+1,rq}+D_{jq+2,rq}=d+2r>d,
\]
contradicting the isomorphisms in \eqref{eq:q_geq_isos}. Likewise, if $r<0$, then we use the fact that $\aredp{rp+j}[D_{jq,rq}]$ must contain an element of grading
\[
d-D_{jq+1,rq}+D_{jq,rq}=d-2r>d
\]
to contradict the isomorphisms in \eqref{eq:q_geq_isos}. This proves condition \eqref{it:q_geq_3} of the lemma.
 
\end{proof}

Putting this all together we obtain the main theorem of this section.
\begin{thm}\label{thm:HFKbound}
\thmHFKbound{}
\end{thm}
\begin{proof}
By Lemma~\ref{lem:ared_recovery} and Proposition~\ref{prop:Vi_recovery}, we can conclude that $V_k(K)=V_k(K')$ and $\ared{k}(K)\cong\ared{k}(K')$ for all $k\geq 0$. The conclusion follows since one can recover the genus, Alexander polynomial and fibredness from this data \cite[Proposition~2.3]{McCoy2020torus_char}.
\begin{enumerate}
 \item The genus is detected as $g(K)=\min\{k\geq 0 \mid V_k+\dim \ared{k}=0\}$.
 \item A knot $K$ of genus $g$ is fibred if and only if $V_{g-1}+\dim\ared{g-1}=1$.
 \item The torsion coefficients of the Alexander polynomial $\Delta_K(t)$ satisfy the formula $t_k(K)= V_k + \chi(\ared{k})$ for all $k$ and the torsion coefficients $t_k$ for $k\geq 0$ are sufficient to determine $\Delta_K(t)$.
\end{enumerate}
\end{proof}

\subsection{Thickness one knots.}
In this section, we exhibit our principal examples of knots with Property~\spliff{}. Recall that a knot $K$ has (knot Floer) thickness at most $n$, if there exists an integer $\rho$ such that the knot Floer homology group $\hfkhat_d(K,s)$ is non-zero only in gradings $s+\rho-n\leq d \leq s+ \rho$. The thickness of $K$ is defined to be the minimal $n$ for which $K$ has thickness at most $n$. 
A knot $K$ is said to have thin knot Floer homology if it has thickness zero, that is, there is an integer $\rho$ such that the group $\hfkhat_d(K,s)$ can be non-zero only if $d=s+\rho$. The principal examples of this are alternating knots which are thin with $\rho = \frac{\sigma}{2}$, where $\sigma$ is the usual knot signature \cite{Ozsvath2003alternating}. Similarly quasi-alternating knots are also known to have thin knot Floer homology \cite{Manolescu2008Khovanov}. More generally we wish to understand when knots with thickness one can have Property~\spliff{}. This is is a class that includes all almost-alternating knots \cite{Lowrance2008width} and all prime knots with at most 12 crossings \cite{Knotinfo}.
 
We will use $\calT_{(d)}(n)$ to denote the submodule of $\tp_{(d)}$ generated by $U^{1-n}$. For $n\leq 0$, the module $\calT_{(d)}(n)$ will be trivial. For $n\geq 1$, the module $\calT_{(d)}(n)$ has dimension $n$ as an $\F$-vector space. For $n=1$, there is an isomorphism of graded $\F[U]$-modules $\F_{(d)}\cong \calT_{(d)}(1)$.
\begin{lem}\label{lem:thickness_one_structure}
Let $K$ be a knot with thickness at most one and let $\rho$ be an integer such that for all $s$, the group $\hfkhat_d(K,s)$ is non-zero only for gradings $d\in \{s+\rho,s+\rho-1\}$. Then for all $k\geq 0$, there exist integers $a,b\geq 0$ such $\ared{k}$ takes one of the following forms
\begin{enumerate}[(a)]
\item\label{it:ared_ceil_type} $\ared{k}\cong \F^{a}_{(k+\rho-1)}\oplus \F^{b}_{(k+\rho-2)}\oplus \calT_{(2k)}\left(\ceilfrac{\rho-k}{2}\right)$
\item\label{it:ared_floor_type} $\ared{k}\cong \F^{a}_{(k+\rho-1)}\oplus \F^{b}_{(k+\rho-2)}\oplus \calT_{(2k)}\left(\floorfrac{\rho-k}{2}\right)$
\end{enumerate}
Furthermore, if $\hfkhat_{\rho+k}(K,k)=0$, then $a=0$. If $\hfkhat_{\rho+k-1}(K,k)=0$, then we can assume that $\ared{k}$ is of type~\eqref{it:ared_ceil_type}.
\end{lem}
Note that types~\eqref{it:ared_ceil_type} and~\eqref{it:ared_floor_type} in Lemma~\ref{lem:thickness_one_structure} are not mutually exclusive as $\ceilfrac{\rho-k}{2}=\floorfrac{\rho-k}{2}$ whenever $\rho-k$ is even.
\begin{proof}
The proof is modelled on \cite[Proof of Theorem~1.4]{Ozsvath2003alternating}.
By \eqref{eq:reduced_complex}, we may assume that the chain complex $C=CFK^\infty$ is such that the group $C\{(i,j)\}$ is supported only in degrees $i+j+\rho$ and $i+j+\rho-1$.

We have an exact sequence of complexes
\[
0\rightarrow C\{i \leq -1\, \text{and}\, j\leq k-1 \}\rightarrow C \rightarrow A^+_k=C\{i\geq 0 \, \text{or}\, j\geq k \} \rightarrow 0.\]
Since the complex $C\{i \leq -1 \,\text{and}\, j\leq k-1 \}$ is non-zero only in degrees $\leq \rho+k-2$, we see that the induced map
\begin{equation}\label{eq:map2}
HF^\infty (S^3)=H_*(C)\rightarrow \Abold_k
\end{equation}
is an isomorphism in degrees $\geq \rho+k$ and an injection in degree $\rho+k -1$.  The image of $\eqref{eq:map2}$ is contained automatically in the tower of $\Abold_k$, so we see that $\ared{k}$ is non-zero only in degrees $\leq \rho+k -1$.

In order to understand the structure of $\ared{k}$ in degrees $\leq \rho+k-2$ we turn to the short exact sequence
\[
0\rightarrow A^+_k=C\{i\geq 0 \, \text{or}\, j\geq k \}\rightarrow C\{i\geq 0 \} \oplus C\{j\geq k \}\rightarrow C\{i\geq 0 \, \text{and}\, j\geq k \} \rightarrow 0.\]
Since the complex $C\{i\geq 0  \, \text{and}\, j\geq k \}$ is non-zero only in degrees $\geq k+\rho-1$, we see that the induced map
\begin{equation}\label{eq:map1}
\Abold_k \rightarrow H_*(C\{i\geq 0 \}) \oplus H_*(C\{j\geq k \})\cong \tp_{(0)} \oplus \tp_{(2k)}
\end{equation}
is an isomorphism in degrees $\leq k+\rho-3$ and a surjection in degree $k+\rho-2$. Since $V_k\geq 0$ and the tower of $\Abold_{k}$ is isomorphic $\tp_{(-2V_k)}$, we may assume that in \eqref{eq:map1} the tower of $\Abold_k$ surjects onto the $\tp_{(0)}$ summand in degrees $\leq k+\rho-2$. Thus we see that $\ared{k}$ surjects onto $\tp_{(2k)}$ in degrees $\leq k+\rho-2$. Thus, $\ared{k}$ is non-zero in all even degrees between $2k$ and $\rho+k-2$. Moreover since \eqref{eq:map1} is a morphism of $\F[U]$-modules, we see that $\ared{k}$ contains a submodule of the form $M=\calT_{(2k)}\left(\floorfrac{\rho-k}{2}\right)$, where $M$ will be trivial if $\rho+k-2<2k$. Moreover, the fact that \eqref{eq:map1} is an isomorphism in degrees $\leq k+\rho-3$ implies that $M$ contains all elements of $\ared{k}$ in degrees $<k+\rho -2$.

Putting this altogether we see that there are integers $a,b\geq 0$ such that the structure of $\ared{k}$ takes one of the desired forms:
\begin{enumerate}[(a)]
\item $\ared{k}\cong \F^{a}_{(k+\rho-1)}\oplus \F^{b}_{(k+\rho-2)}\oplus \calT_{(2k)}\left(\ceilfrac{\rho-k}{2}\right)$
\item $\ared{k}\cong \F^{a}_{(k+\rho-1)}\oplus \F^{b}_{(k+\rho-2)}\oplus \calT_{(2k)}\left(\floorfrac{\rho-k}{2}\right)$.
\end{enumerate}
Here type~(a) occurs since the submodule $M$ described above may not occur as an $\F[U]$-module summand of $\ared{k}$. This happens if there is an element of degree $k+\rho-1$ which is not killed by the $U$-action.

Now we refine these statements if $\hfkhat_{\rho+k}(K,k)=0$ or $\hfkhat_{\rho+k-1}(K,k)=0$.

 The chain group of $C\{i \leq -1 \,\text{and}\, j\leq k-1 \}$ in degree $\rho+k-2$ is isomorphic to $\hfkhat_{\rho+k}(K,k)$. So if $\hfkhat_{\rho+k}(K,k)=0$, then the map in \eqref{eq:map2} is actually an isomorphism in degrees $\geq \rho+k-1$. In particular, if $\hfkhat_{\rho+k}(K,k)=0$, then $\ared{k}$ is zero in degree $\rho+k -1$ as well implying $a=0$ in the given forms.

The chain group of $C\{i\geq 0 \, \text{and}\, j\geq k \}$ in degree $k+\rho-1$ is isomorphic to $\hfkhat_{\rho+k-1}(K,k)$. Therefore, if $\hfkhat_{\rho+k-1}(K,k)=0$, the map \eqref{eq:map1} is an isomorphism in degrees $\leq k+\rho-2$ and a surjection in degree $k+\rho-1$. Thus, we get a surjection of $\ared{k}$ onto $\tp_{(2k)}$ in all degrees $\leq k+\rho-1$. This means that $\ared{k}$ is supported in all even degrees between $2k$ and $\rho+k-1$ and that, in fact, the module $M$ from above is contained in a submodule of the form $\calT_{(2k)}\left(\ceilfrac{\rho-k}{2}\right)$. This implies that $\ared{k}$ is of the form $(a)$, as required.
\end{proof}

We will see that for knots with thickness one, we only need to consider the module $\ared{\rho-3}$ to verify whether the knot has Property~\spliff{}.
\begin{lem}\label{lem:thickness_one_characterization}
Let $K$ be a knot with thickness at most one and let $\rho$ be an integer such that for all $s$, the group $\hfkhat_d(K,s)$ is non-zero only for gradings $d\in \{s+\rho,s+\rho-1\}$. Then $K$ has Property~\spliff{} if and only if $\rho \leq 2$ or $\ared{\rho-3}$ has Property~\spliff{}.
\end{lem}
\begin{proof}
Suppose that $\ared{k}$ does not have Property~\spliff{} for some $k\geq 0$. Lemma~\ref{lem:thickness_one_structure} shows that for $k\geq 0$ the only way $\ared{k}$ can fail to have Property~\spliff{} is if the term $\calT_{(2k)}\left(\ceilfrac{\rho-k}{2}\right)$ or $\calT_{(2k)}\left(\floorfrac{\rho-k}{2}\right)$ is an $\F$-summand supported in degree $2k$ where $2k\neq k+\rho-1$ and $2k\neq k+\rho-2$.

Thus we have $\ceilfrac{\rho-k}{2}=1$ or $\floorfrac{\rho-k}{2}=1$ where $k\not\in\{\rho-1, \rho-2\}$.

However, if $\ceilfrac{\rho-k}{2}=1$, then $k=\rho-1$ or $k=\rho-2$. So the remaining possibility is that $\floorfrac{\rho-k}{2}=1$. This happens only if $k=\rho-2$ or $k=\rho-3$. In conclusion, $\ared{k}$ has Property~\spliff{} for all $k\geq 0$ with the possible exception of $k=\rho-3$. Thus we see that $K$ does not have Property~\spliff{} if and only if $\rho\geq 3$ and $\ared{\rho-3}$ does not have Property~\spliff{}.
\end{proof}
In practice, this allows us to find many examples of knots with Property~\spliff{} via the following proposition.
\propthicknessone*

\begin{proof}
We establish first the conditions for $K$ to have Property~\spliff{}. Lemma~\ref{lem:thickness_one_characterization} says immediately that $K$ has Property~\spliff{} if $\rho\leq 2$. Thus we may suppose $\rho\geq 3$.

If $\hfkhat_{2\rho-3}(K,\rho-3)=0$, then Lemma~\ref{lem:thickness_one_structure} shows that there is an integer $b\geq 0$ such that
\[\text{$\ared{\rho-3}\cong \F^{b}_{(2\rho-5)}\oplus \F_{(2\rho-6)}$ or $\ared{\rho-3}\cong \F^{b}_{(2\rho-5)}\oplus \calT_{(2\rho-6)}(2)$}.\]
In either event, $\ared{\rho-3}$ has Property~\spliff{} and so Lemma~\ref{lem:thickness_one_characterization} implies that $K$ also has Property~\spliff{}.

If $\hfkhat_{2\rho-4}(K,\rho-3)=0$, then Lemma~\ref{lem:thickness_one_structure} shows that there are integer s $a,b\geq 0$ such that
\[\ared{\rho-3}\cong \F^{a}_{(2\rho-4)}\oplus \F^{b}_{(2\rho-5)}\oplus \calT_{(2\rho-6)}(2).\]
Thus $\ared{\rho-3}$ and hence also $K$ have Property~\spliff{}.

Now we turn our attention to the conditions for $mK$ to have Property~\spliff{}. These follow immediately from the fact that there is an isomorphism \cite[Proposition~3.7]{Ozsvath2004knotinvariants}
\[
\hfkhat_{d}(mK,s)\cong \hfkhat_{-d}(K,-s).
\]
In particular, if $K$ has thickness one with knot Floer homology supported in degrees $s+\rho$ and $s+\rho-1$, then $mK$ also has thickness at most one supported only in degrees $s-\rho$ and $s-\rho+1$. Thus, the conditions for $mK$ are obtain by replacing $\rho$ with $1-\rho$ in the conditions for $K$.
\end{proof}

\section{Applications to characterizing slopes}\label{sec:applications}
We now put together the results of the preceding sections to prove the remaining main results.

\interestingclasses*
\begin{proof}For all of the three classes we will show that that both the knot $K$ and its mirror $mK$ have Property~\spliff{}.
\begin{enumerate}[(i)]
\item Suppose that $K$ has thin knot Floer homology. Since $mK$ also has thin knot Floer homology, it suffices to show that $K$ has Property~\spliff{}. However this follows easily from Proposition~\ref{prop:thickness_one}. Suppose that there is $\rho\in\Z$ such that for all $s$ the group $\hfkhat_d(K,s)$ is non-zero only if $d=s+\rho$. The group $\hfkhat(K,\rho-3)$ is supported only in grading $d=2\rho-3$ implying that $\hfkhat_{2\rho-4}(K,\rho-3)=0$.

\item Suppose that $K$ is an $L$-space knot. Since $\ared{k}(K)=0$ for all $k$, we see that $K$ has Property~\spliff{}. It remains to show that $mK$ has Property~\spliff{}. Fix some integer $n\geq 2g(K)-1$. As calculated in Lemma~\ref{lem:redHFp_small_g}, we have that $\ared{k}(mK)$ is isomorphic to (an appropriately shifted) $\hfred(S_{mK}^3(n),k)$ for $0\leq k<g(K)$. If $mK$ does not have Property~\spliff{}, then the exact triangle relating $\hfhat$ and $\hfp$ would imply that $\dim \hfhat_{\rm red}(S_{mK}^3(n),\spincs)>2$ for some \spinc-structure $\spincs$. However, if $\spincs$ is a \spinc-structure on $S^3_K(-n)$, then  Lemma~18 and Proposition~19 in \cite{Gainullin2017mapping} show that
\[\hfred(S^3_K(-n),\spincs)\cong \calT(m_\spincs),\]
for some integer $m_\spincs\geq 0$. In particular, the exact triangle relating $\hfhat$ and $\hfp$ shows that $\dim \hfhat_{\rm red}(S^3_K(-n),\spincs)\leq 2$.  However the dimension of $\hfhat$ is invariant under orientation-reversal \cite[Proposition~2.5]{Ozsvath2004applications}, so the fact that $S_{mK}^3(n)\cong -S_{mK}^3(-n)$ implies that $\dim \hfhat_{\rm red}(S_{mK}^3(n),\spincs)\leq 2$ for all \spinc-structures. Thus we have shown that $mK$ also has Property~\spliff{}. %With a little more care, one can actually establish that $\ared{k}(mK)\cong \calT(V_k(K))$ for all $k\geq 0$.

\item For each prime knot $K$ with at most 12 crossings KnotInfo contains $\hfkhat(K)$ encoded in the form of its Poincaré polynomial \cite{Knotinfo}. This allows us to see for which $s$ and $d$ the groups $\hfkhat_d(K,s)$ are non-trivial. From this data one can easily determine that all the non-alternating prime knots with at most 12 crossings have knot Floer homology of thickness at most one. As discussed above, the knots with thin knot Floer homology satisfy Conjecture~\ref{conj:main}. For the knots with knot Floer homology of thickness one, one can calculate the value of $\rho$ occurring in Proposition~\ref{prop:thickness_one} confirm whether both conditions (i) and (ii) of Proposition~\ref{prop:thickness_one} are satisfied. This is sufficient to show that $K$ and $mK$ have Property~\spliff{} for all prime knots with at most 12 crossings except those listed in Table~\ref{table:awkward_examples}.
\end{enumerate}
\end{proof}

\subsection{Integral slopes on almost $L$-space knots}
We now turn our attention to characterizing slopes for almost $L$-space knots.

\begin{lem}\label{lem:almost_properties}
Let $K$ be an almost $L$-space knot. Then
\begin{enumerate}[(i)]
\item $K$ has Property~\spliff{} and
\item either $g(K)=\nu^+(K)$ or $g(K)= 1$.
\end{enumerate}
\end{lem}
\begin{proof}
If a rational homology sphere $Y$ is an almost $L$-space, then there is a unique \spinc-structure $\spincs_0$ such that $\hfred(Y,\spincs_0)$ is non-zero and that there is an integer $m\geq 1$ such that $\hfred(Y,\spincs_0)\cong \calT(m)$ \cite[Proposition~3.1]{Binns2019almostLspace}. Proposition~\ref{prop:redHFp} shows that for all $n\geq 2g(K)-1$ and all $k\in\Z$, the group $\ared{k}$ occurs as a summand in $\hfred(S_K^3(n), i)$ for some $i$. On the other hand, if $\ared{k}$ is non-zero for some $k>0$, then $\ared{-k}$ is also non zero \cite[Lemma~2.3]{Hom2015reducible}. Therefore if $K$ is an almost $L$-space knot, then $\ared{0}\cong \calT(m)$ for some $m\geq 1$ and $\ared{k}=0$ for all $k> 0$. In particular, an almost $L$-space knot has Property~\spliff{}.
Furthermore, Proposition~\ref{prop:ared_properties} implies that if $\nu^+(K)<g(K)$, then $\ared{g(K)-1}$ is non-zero. However for an almost $L$-space knot, $\ared{k}$ is non-zero only for $k=0$. It follows that $g(K)=\nu^+(K)$ or $g(K)= 1$.
\end{proof}
%\footnote{Although it is not relevant for this paper, the almost L-space knots with $g(K)=1$ are precisely the figure eight knot $4_1$, the left-handed trefoil $T_{3,-2}$ and $\overline{5_2}$ \cite{Baldwin52}.}
We establish characterizing slopes for $T_{2,3}\# T_{2,3}$ first. This plays a special role in the analysis as the unique non-prime almost $L$-space knot \cite{Binns2019almostLspace}.
\begin{prop}\label{prop:trefoil_sum}
Any integer slope $n\geq 32$ is characterizing for $T_{2,3}\# T_{2,3}$.
\end{prop}
\begin{proof}
Let $K=T_{2,3}\# T_{2,3}$ and suppose that $K'$ is a knot such that $S_{K'}^3(n)\cong S_{K}^3(n)$ for some $n\geq 32$. Proposition~\ref{prop:Vi_recovery} applies to show that  $\nu^+(K')=\nu^+(K)=2$. Furthermore, since $S_{K'}^3(n)$ is an almost $L$-space and $n>2\nu^+(K')-1$, this implies that $K'$ is also an almost $L$-space knot. In particular, $g(K')=g(K)=2$.

The JSJ decomposition of $S^3_K$ comprises two trefoil complements glued to a composing space. Since a composing space becomes a copy of $T^2\times [0,1]$ under integer Dehn filling, the manifold $S_{K}^3(n)$ is obtained by gluing together two trefoil complements. Thus the JSJ decomposition of $S_{K}^3(n)\cong S_{K'}^3(n)$ contains only Seifert fibred pieces with all exceptional fibres being of orders two or three.

Next we consider the four possibilities for $K'$ allowed by Theorem~\ref{thm:knot_JSJ_decomp}: $K'$ is either of hyperbolic type, a torus knot, a cable knot or a composite knot.

If $K'$ were of hyperbolic type, then Proposition~\ref{prop:length_bound} would show that
\[
\ell_{K'}(n)\geq \frac{n}{\sqrt{3}(2g(K')-1)}\geq \frac{32}{3\sqrt{3}}>6.
\]
This would imply, via the 6-theorem, that the outermost piece of $S_{K'}^3$ would remain hyperbolic after the Dehn filling, contradicting the fact that the JSJ decomposition of $S_{K'}^3(n)$ contains only Seifert fibred pieces.

If $K'$ were a cable knot or a torus knot, then Proposition~\ref{prop:JSJ_surgery_torus} or Proposition~\ref{prop:JSJ_surgery_cable} would show that the JSJ decomposition of $S_{K'}(n)$ would contain a Seifert fibred piece with an exceptional fibre of order at least 
\[n-4g(K')-2=n-10\geq 22,\]
contradicting the fact that the maximal order of an exceptional fibre is three.

Thus the only remaining possibility is that $K'$ is a composite knot. However $T_{2,3}\# T_{2,3}$ is the unique non-prime almost $L$-space knot \cite{Binns2019almostLspace}, so this implies that $K\simeq K'$ and hence that $n$ is a characterizing slope.
\end{proof}

With this case in hand, we establish characterizing slopes for almost $L$-space knots in general.

\almostLspaceknot*
\begin{proof}
Since $L$-space knots and almost $L$-space knots have Property~\spliff{}, Theorem~\ref{thm:spaff_slopes} shows that we only need to consider integer slopes. Thus let $K$ be a non-trivial $L$-space knot or an almost $L$-space knot and suppose that some $n\geq 4g(K)+14$ is a non-characterizing slope for $K$. Thus we may take $K'$ be a knot distinct from $K$ such that $S_{K}^3(n)\cong S_{K'}^3(n)$. Proposition~\ref{prop:Vi_recovery} applies to show that $\nu^+(K)=\nu^+(K')$. Since $S_{K}^3(n)$ is an $L$-space or an almost L-space, it follows that $K'$ must also be an $L$-space knot or an almost $L$-space knot. By Lemma~\ref{lem:almost_properties}(ii), an almost $L$-space knot satisfies $\nu^+(K)=g(K)\geq 1$ unless $\nu^+(K)=0$ and $g(K)=1$. For an $L$-space knot we always have $\nu^+(K)=g(K)$.  In either case, the equality $\nu^+(K)=\nu^+(K')$ implies that $g(K)=g(K')$. Furthermore, given Proposition~\ref{prop:trefoil_sum} and the fact that $T_{3,2}\# T_{3,2}$ is the unique composite knot which is an $L$-space knot or an almost $L$-space knot \cite{Binns2019almostLspace, Krcatovich2015reduced}, we may assume that both $K$ and $K'$ are prime. Therefore Theorem~\ref{thm:bounded_genus} applies to yield an upper bound for $n$.
\end{proof}

\bibliographystyle{alpha}
\bibliography{master}

\end{document}